\documentclass[eqthmnum]{jt-calcs}
\usepackage{mathrsfs}
\usepackage[backend=bibtex,style=alphabetic,sorting=nyt]{biblatex}
\bibliography{bibliography}
\AtEveryBibitem{%
  \clearfield{pagetotal}%
}

\usepackage{tikz-cd}

\usepackage{kbordermatrix}
\usepackage{booktabs}
\usepackage{multirow}

\newcommand{\Spr}{\ensuremath{\mathsf{Spr}}}

\newcommand{\Iso}{\ensuremath{\mathrm{Iso}}}
\newcommand{\Cusp}{\ensuremath{\mathrm{Cusp}}}

\newcommand{\Frob}{\ensuremath{\mathrm{Frob}}}

\newcommand{\sbpair}[2]{\ensuremath{\bigl(\begin{smallmatrix} #1 \\ #2 \end{smallmatrix}\bigr)}}
\newcommand{\ssymb}[2]{\ensuremath{\bigl[\begin{smallmatrix} #1 \\ #2 \end{smallmatrix}\bigr]}}
\newcommand{\bpair}[2]{\ensuremath{\begin{pmatrix} #1 \\ #2 \end{pmatrix}}}
\newcommand{\symb}[2]{\ensuremath{\begin{bmatrix} #1 \\ #2 \end{bmatrix}}}

\setlist[enumerate,1]{label=(\alph*)}

\title{Action of Automorphisms on Irreducible Characters of Symplectic Groups}
\author{Jay Taylor}
\address{Department of Mathematics, University of Arizona, 617 N. Santa Rita Ave., Tucson AZ 85721, United States.}
\email{jaytaylor@math.arizona.edu}
\mscno{2010}{20C30}{20C15}
\keywords{Finite groups, automorphisms, symplectic groups, McKay conjecture, generalised Gelfand--Graev representations.}

\begin{document}
\begin{abstract}
Assume $G$ is a finite symplectic group $\Sp_{2n}(q)$ over a finite field $\mathbb{F}_q$ of odd characteristic. We describe the action of the automorphism group $\Aut(G)$ on the set $\Irr(G)$ of ordinary irreducible characters of $G$. This description relies on the equivariance of Deligne--Lusztig induction with respect to automorphisms. We state a version of this equivariance which gives a precise way to compute the automorphism on the corresponding Levi subgroup; this may be of independent interest. As an application we prove that the global condition in Sp\"ath's criterion for the inductive McKay condition holds for the irreducible characters of $\Sp_{2n}(q)$.
\end{abstract}

\section{Introduction}
\begin{pa}
The representation theory of finite groups is abound with many deep and fascinating conjectures nicknamed local/global conjectures; the paradigm of these is the McKay Conjecture. In a landmark paper \cite{isaacs-malle-navarro:2007:a-reduction-theorem-for-the-mckay-conj} Isaacs--Malle--Navarro showed that the McKay Conjecture holds for all finite groups if a list of (stronger) conditions, jointly referred to as the inductive McKay condition, holds for the universal covering group of each finite simple group. We note that in the wake of \cite{isaacs-malle-navarro:2007:a-reduction-theorem-for-the-mckay-conj} several other major local/global conjectures have also been reduced to checking certain inductive conditions.
\end{pa}

\begin{pa}
Showing that these inductive conditions hold has revealed itself to be a difficult problem. One of the main difficulties arises from the fact that one needs some knowledge of how the automorphism group $\Aut(G)$, of a quasisimple group $G$, acts on the set of irreducible characters $\Irr(G)$. When the simple quotient $G/Z(G)$ is a group of Lie type then this question has turned out to be surprisingly vexing considering the large amount of machinery at our disposal. The main result of this paper gives a complete solution to this problem when $G = \Sp_{2n}(q)$ and $q$ is a power of an odd prime. Note that the corresponding statement when $q$ is even is an easy consequence of known results from Lusztig's classification of irreducible characters \cite{lusztig:1984:characters-of-reductive-groups}, see also \cite{digne-michel:1990:lusztigs-parametrization} and \cite{cabanes-spaeth:2013:equivariance-connected-centre}.
\end{pa}

\begin{pa}
In \cite[Theorem 2.12]{spaeth:2012:inductive-mckay-defining} Sp\"ath gave a version of the inductive McKay condition which is specifically tailored to the finite groups of Lie type. There are several conditions in this statement, some of which are global and some of which are local. As an application of our result we show that the global condition concerning the stabilisers of irreducible characters of $G$ under automorphisms holds for finite symplectic groups over fields of odd characteristic, see \cref{prop:mckay}. We note that Cabanes and Sp\"ath \cite{cabanes-spaeth:2016:inductive-mckay-symplectic} have recently shown the whole inductive McKay condition holds for the symplectic groups, hence independently proving \cref{prop:mckay} using completely different methods.
\end{pa}

\begin{pa}
To state our result we need to introduce some notation. Let $\bG$ be a connected reductive algebraic group defined over an algebraic closure $\overline{\mathbb{F}}_p$ of the finite field $\mathbb{F}_p$ of prime order $p$. We assume $F :\bG \to \bG$ is a Frobenius endomorphism endowing $\bG$ with an $\mathbb{F}_q$-rational structure $\bG^F$. If $\bG^{\star F^{\star}}$ is a group dual to $\bG^F$ then to each semisimple element $s \in \bG^{\star F^{\star}}$ we have a corresponding Lusztig series $\mathcal{E}(\bG^F,s) \subseteq \Irr(\bG^F)$ of irreducible characters. Moreover, we have $\Irr(\bG^F) = \bigsqcup_{[s]} \mathcal{E}(\bG^F,s)$ is a disjoint union of these series where the union runs over all the $\bG^{\star F^{\star}}$-conjugacy classes of semisimple elements. As we will see, a particularly important role is played by those series $\mathcal{E}(\bG^F,s)$ for which $s$ is a quasi-isolated in $\bG^{\star}$. We recall that this means $C_{\bG^{\star}}(s)$ is not contained in any proper Levi subgroup of $\bG^{\star}$.
\end{pa}

\begin{pa}
If $\Aut(\bG^F)$ denotes the automorphism group of $\bG^F$ then we set ${}^{\sigma}\chi = \chi\circ\sigma^{-1}$ for any $\sigma \in \Aut(\bG^F)$ and $\chi \in \Irr(\bG^F)$; this defines an action of $\Aut(\bG^F)$ on $\Irr(\bG^F)$. The automorphism group $\Aut(\bG^F)$ is well known to be generated by inner, diagonal, field, and graph automorphisms. In the case of diagonal automorphisms the action on $\Irr(\bG^F)$ is well understood by work of Lusztig \cite{lusztig:1988:reductive-groups-with-a-disconnected-centre}. Moreover, in the case of symplectic groups defined over a field of odd characteristic there are no graph automorphisms. With this in place we may state our main result.
\end{pa}

\begin{thm}\label{thm:main-theorem}
Assume $\bG^F = \Sp_{2n}(q)$ with $q$ odd. If $\sigma \in \Aut(\bG^F)$ is a field automorphism and $\chi \in \mathcal{E}(\bG^F,s)$ is an irreducible character, with $s \in \bG^{\star F^{\star}}$ quasi-isolated in $\bG^{\star}$, then we have ${}^{\sigma}\chi = \chi$.
\end{thm}

\begin{pa}
Let us now outline our strategy for proving \cref{thm:main-theorem}. The series $\mathcal{E}(\bG^F,s)$ is broken up into a disjoint union of Harish-Chandra series. As part of their proof of the McKay conjecture for odd degree characters \cite{malle-spaeth:2016:characters-of-odd-degree} Malle and Sp\"ath have shown how to describe the action of $\sigma$ on the constituents of a Harish-Chandra series. This involves understanding the effect of $\sigma$ on cuspidal characters and on characters of the relative Weyl group. In the case of symplectic groups the crucial point is to determine the action of $\sigma$ on the cuspidal characters because the action of $\sigma$ on the characters of the relative Weyl group is quite simple to understand.
\end{pa}

\begin{pa}
To understand the action on a cuspidal character $\chi \in \mathcal{E}(\bG^F,s)$ we use Kawanaka's generalised Gelfand--Graev representations (GGGRs). Specifically we choose a regular embedding $\bG \to \widetilde{\bG}$ into a group with connected centre which is defined over $\mathbb{F}_q$. The cuspidal character $\chi$ then occurs in the restriction of a cuspidal character $\widetilde{\chi} \in \Irr(\widetilde{\bG}^F)$. Associated to $\widetilde{\chi}$ we have a corresponding wave front set $\mathcal{O}_{\widetilde{\chi}}^* \subseteq \widetilde{\bG}$, see \cite{kawanaka:1985:GGGRs-and-ennola-duality,lusztig:1992:a-unipotent-support,taylor:2016:GGGRs-small-characteristics}, which is an $F$-stable unipotent conjugacy class of $\widetilde{\bG}$. Moreover to each rational element $u \in \mathcal{O}_{\widetilde{\chi}}^{*F} \subseteq \widetilde{G}$ we have a corresponding GGGR $\Gamma_u^{\widetilde{\bG}^F}$.
\end{pa}

\begin{pa}
It is known that the multiplicity of $\widetilde{\chi}$ in $\Gamma_u^{\widetilde{\bG}^F}$ is small, see \cite[15.4]{taylor:2016:GGGRs-small-characteristics} for a precise statement. The key ingredient to our proof is that there exists an element $u$ for which this multiplicity is $1$. This crucial property allows us to relate the action of $\sigma$ on $\chi$ to the action of $\sigma$ on the GGGRs of $\bG^F$, which is easy to describe. In particular, we deduce that ${}^{\sigma}\chi = \chi$ is always fixed by $\sigma$. This, combined with the Harish-Chandra techniques mentioned above, allows us to deduce the theorem for the whole series $\mathcal{E}(\bG^F,s)$.
\end{pa}

\begin{pa}
To analyse the effect of $\sigma$ on the constituents of any Lusztig series we use the fact that an irreducible character can be obtained as the Deligne--Lusztig induction of an irreducible character from a Levi subgroup of $\bG^F$. It is well known that Deligne--Lusztig induction is equivariant with respect to isogenies, see \cite[13.22]{digne-michel:1991:representations-of-finite-groups-of-lie-type} for instance. In the case of bijective isogenies one can give a shorter proof, such as that given by Bonnaf\'e in \cite[\S2]{navarro-tiep-turull:2008:brauer-characters-with-cyclotomic}. Unfortunately these statements are not sufficient for our purpose as they do not allow us to control the type of the automorphism induced by the isogeny upon restricting to the Levi subgroup. Here we obtain a version of the equivariance, for bijective isogenies, which allows us to explicitly compute the resulting automorphism, see \cref{thm:equivariance-DL-induction} for a precise statement. This result may be of independent interest. Specifically, using \cref{thm:equivariance-DL-induction} we may reduce to the case of looking at irreducible characters of Levi subgroups contained in a Lusztig series labelled by semisimple elements which are quasi-isolated in the dual of the Levi; thus putting us in the setting of \cref{thm:main-theorem}. We apply precisely this process to prove the global part of Sp\"ath's criterion for the inductive McKay condition, see \cref{prop:mckay}.
\end{pa}

\begin{pa}
Let us now make some comments concerning possible generalisations of this result to other series of groups. Our overall strategy is suited to tackling other series, although some parts would arguably be more involved in other cases. For instance, understanding the action of automorphisms on Harish-Chandra series will be more complicated because the structure of Levi subgroups is more complicated. The main problem in generalising to other series is proving the multiplicity $1$ statement for GGGRs and cuspidal characters mentioned above. With this in mind let $\widetilde{\bG}$ be a connected reductive algebraic group with connected centre whose quotient is $\widetilde{\bG}/Z(\widetilde{\bG})$ is simple. When the quotient is of type $\B_n$ or $\D_n$ we have checked that this multiplicity $1$ condition holds for cuspidal characters, covering cuspidal characters in quasi-isolated series, except when the centraliser $C_{\widetilde{G}^{\star}}(s)$ has a twisted type $\A$ factor of large rank. In this case the analogue of \cref{prop:numerical-trick} can fail and so we cannot use this sufficient condition to obtain the multiplicity $1$ statement. We hope to settle this problem in the future.
\end{pa}

\begin{pa}
The structure of the paper is as follows. In \cref{sec:aut-rat-orb} we recall some general statements showing how to describe the action of automorphisms on rational orbits under the action of a connected reductive algebraic group equipped with a Steinberg endomorphism. \Cref{sec:isogenies,sec:duality} recall some fundamental results concerning isogenies and duality. In \crefrange{sec:duality-Lusztig-series}{sec:equivariance-twisted-induction} we describe the basic equivariance statements needed for reducing our study of automorphisms to Lusztig series labelled by quasi-isolated semisimple elements. In \cref{sec:regular-embeddings} we consider regular embeddings. Note that the results from these sections are well known but are not usually phrased in this language. This rephrasing is critical as it facilitates our inductive argument; a similar approach is taken in \cite{lusztig:1984:characters-of-reductive-groups}. The key result from these sections is \cref{thm:equivariance-DL-induction}, which gives the equivariance of Deligne--Lusztig induction. We note that, up to this point, we have tried to state these results in the widest possible generality as we feel they may be useful elsewhere.
\end{pa}

\begin{pa}
In \Cref{sec:GGGRs} we show that GGGRs are equivariant with respect to Frobenius endomorphisms and automorphisms of algebraic groups, see \cref{prop:invariant-GGGRs}. At this point we must assume that $p$ is a good prime to ensure that the GGGRs are defined. The crucial multiplicity $1$ statement for cuspidal characters is proved in \cref{sec:mult-1}, see \cref{thm:cuspidal-fixed}. This uses the combinatorics of the generalised Springer correspondence described in \cref{sec:springer}. We then prove our main result, \cref{thm:main-theorem}, in \cref{sec:quasi-isolated-series}, see \cref{prop:quasi-iso-fixed}. In the final section we prove the global part of Sp\"ath's criterion for the inductive McKay condition.
\end{pa}

\begin{acknowledgments}
We thank Gunter Malle and Britta Sp\"ath for their valuable comments on an earlier version of this article. The author gratefully acknowledges the financial support of an INdAM Marie-Curie Fellowship and grants CPDA125818/12 and 60A01-4222/15 of the University of Padova.
\end{acknowledgments}

\section{Notation}
\begin{assumption}
From this point forward $p$ and $\ell$ will denote distinct prime numbers.
\end{assumption}

\begin{pa}\label{pa:T-Z-sets}
Let us assume $G$ is a group and $X$ is a (left) $G$-set with action map $\cdot : G \times X \to X$. The orbits of $G$ acting on $X$ will be denoted by $X/G$ and if $x \in X$ then we denote by $[x] \subseteq X$ the $G$-orbit containing $x$. For any subset $Y \subseteq X$ we consider the centraliser $C_G(Y) = \{g \in G \mid g\cdot y = y$ for all $y\in Y\}$ and normaliser $N_G(Y) = \{g \in G \mid g\cdot Y = Y\}$ of $Y$ under the action of $G$; here $g\cdot Y = \{g\cdot y \mid y \in Y\}$. Now assume $\phi$ is a map $X \to X$ then we define sets
\begin{align*}
T_G(Y,\phi) &= \{g \in G \mid g\cdot\phi(Y) = Y\},\\
Z_G(Y,\phi) &= \{g \in G \mid Y = \phi(g\cdot Y)\}.
\end{align*}
The first set is nothing other than the usual transporter, which is either empty or a right coset of $N_G(Y)$ in $G$. If we assume $\phi$ is a bijection then $Z_G(Y,\phi)$ is either empty or a left coset of $N_G(Y)$ in $G$. If $H \leqslant N_G(Y)$ is a subgroup of the normaliser then we denote by $H\backslash T_G(Y,\phi)$, resp., $Z_G(Y,\phi)/H$, the right, resp., left, cosets of $H$ in $G$ contained in $T_G(Y,\phi)$, resp., $Z_G(Y,\phi)$. For $x \in X$ we set $C_G(x) = C_G(\{x\})$, $N_G(x) = N_G(\{x\})$, $T_G(x,\phi) = T_G(\{x\},\phi)$, and $Z_G(x,\phi) = Z_G(\{x\},\phi)$.
\end{pa}

\begin{pa}
If $\mathcal{P}(G)$ is the power set of $G$ then we have an action of $G$ on $\mathcal{P}(G)$ by setting ${}^gS = gSg^{-1}$ for any $g \in G$ and $S \in \mathcal{P}(G)$. If $Y \in \mathcal{P}(G)$ is a subset then $C_G(Y)$, $N_G(Y)$, $T_G(Y,\phi)$, and $Z_G(Y,\phi)$ are defined with respect to this action unless otherwise specified. We will also write $S^g$ for $g^{-1}Sg$. If $H \leqslant G$ is a subgroup then we will denote by $W_G(H)$ the section $N_G(H)/H$.
\end{pa}

\begin{pa}
Ignoring any additional structure we will denote by $\Aut(G)$ the automorphism group of $G$ in the category of groups. Moreover, if $g \in G$ is an element then we will denote by $\imath_g \in \Aut(G)$ the inner automorphism defined by $\imath_g(x) = gxg^{-1}$ for all $x \in G$; this defines a homomorphism $G \to \Aut(G)$. If $\phi : G \to G$ is a homomorphism then we will denote by $\Aut(G,\phi) = C_{\Aut(G)}(\phi) \leqslant \Aut(G)$ the subgroup of automorphisms commuting with $\phi$. Note that if $G^{\phi} = \{g \in G \mid \phi(g) = g\}$ is the fixed point subgroup then the natural inclusion $G^{\phi} \to G$ induces a homomorphism $\Aut(G,\phi) \to \Aut(G^{\phi})$.
\end{pa}

\begin{pa}
We assume chosen an algebraic closure $\Ql$ of the $\ell$-adic field $\mathbb{Q}_{\ell}$. Let us assume now that $G$ is a finite group then we will denote by $\Irr(G)$ the set of ordinary irreducible characters $\chi : G \to \Ql$. If $\phi : G \to H$ is an isomorphism of finite groups then the map $\Irr(G) \to \Irr(H)$ defined by $\chi \mapsto {}^{\phi}\chi := \chi \circ \phi^{-1}$ is a bijection. If $M$ is a right, resp., left, $\Ql G$-module then we denote by ${}^\phi M$, resp., $M^{\phi}$, the vector space $M$ viewed as a $\Ql H$-module by letting $h \in H$ act as $\phi^{-1}(h)$ on $M$. Note that if $M$ affords the character $\chi$ then ${}^{\phi}M$ affords the character ${}^{\phi}\chi$. If $\phi$ is a conjugation map $\imath_g$ then we will simply write ${}^g\chi$, resp., ${}^gM$, $M^g$, for ${}^{\phi}\chi$, resp., ${}^{\phi}M$, $M^{\phi}$.
\end{pa}

\begin{pa}
For any $g \in G$ or $h \in H$ we will denote by $\phi g : G \to H$ and $h\phi : G \to H$ the isomorphism $\phi\imath_g : G \to H$ and $\imath_h\phi : G \to H$ respectively. This is consistent with the notation above. This may clearly lead to confusion over what $\phi g(x)$ and $h\phi(x)$ stand for when $x \in G$. To avoid confusion we will always write ${}^h\phi(x)$ or $\phi({}^gx)$ for $\imath_h\phi(x)$ and $\phi\imath_g(x)$ when we apply this map to elements.
\end{pa}

\begin{pa}\label{pa:stab-notation}
If $H = G$ in the above construction then the map $\Irr(G) \to \Irr(G)$ defined by $\chi \mapsto {}^{\phi}\chi$ defines a left action of $\Aut(G)$ on $\Irr(G)$. If $\widetilde{G}$ is a group equipped with a homomorphism $\varphi : \widetilde{G} \to \Aut(G)$ then we define an action of $\widetilde{G}$ on $\Irr(G)$ via the induced action from $\Aut(G)$. Moreover, we denote by $\widetilde{G}_{\chi} \leqslant \widetilde{G}$ the stabiliser of $\chi$ under this action.
\end{pa}

\section{Automorphisms and Rational Orbits}\label{sec:aut-rat-orb}
\begin{assumption}
From this point forward $\mathbb{K} = \overline{\mathbb{F}}_p$ will denote a fixed algebraic closure of the finite field $\mathbb{F}_p$ of cardinality $p$. Moreover, we will write $\mathbb{G}_m$, resp., $\mathbb{G}_a$, for the set $\mathbb{K}\setminus \{0\}$, resp., $\mathbb{K}$, viewed as a multiplicative, resp., additive, algebraic group.
\end{assumption}

\begin{pa}\label{pa:setup}
Let $\bG$ be a connected algebraic $\mathbb{K}$-group. We will assume that $\bG$ is endowed with a Steinberg endomorphism $F$, i.e., an endomorphism for which some power of $F$ is a Frobenius endomorphism. We now assume that $\mathcal{O}$ is a set on which $\bG$ acts transitively. Moreover, we assume that $\mathcal{O}$ is equipped with a map $F' : \mathcal{O} \to \mathcal{O}$ such that the following hold:
\begin{enumerate}[label=(F\arabic*)]
	\item $F'(g\cdot x) = F(g)\cdot F'(x)$ for all $g \in \bG$ and $x \in \mathcal{O}$
	\item the stabiliser $\Stab_{\bG}(x) \leqslant \bG$ of any point $x \in \mathcal{O}$ is a closed subgroup.
\end{enumerate}
Note, this is the setup of \cite[4.3.1]{geck:2003:intro-to-algebraic-geometry}. With these assumptions the finite group $\bG^F$ acts on the set of fixed points $\mathcal{O}^{F'}$.
\end{pa}

\begin{pa}\label{pa:rational-orbits}
We wish to now recall a parameterisation of the orbits $\mathcal{O}^{F'}/\bG^F$. With this in mind we denote by $A_{\bG}(x)$ the finite component group $\Stab_{\bG}(x)/\Stab_{\bG}^{\circ}(x)$ of the stabiliser of $x \in \mathcal{O}$. We will write $\overline{\phantom{x}} : \Stab_{\bG}(x) \to A_{\bG}(x)$ for the natural projection map. As $\bG$ acts transitively on $\mathcal{O}$ we have the set $T_{\bG}(x_0,F') \neq \emptyset$ for any $x_0 \in \mathcal{O}$ and moreover we have an action of $\Stab_{\bG}(x_0)$ on $T_{\bG}(x_0,F')$ given by $a\cdot z = a^{-1}zF(a)$. This induces an action of $A_{\bG}(x_0)$ on $A_{\bG}(x_0,F') = \Stab_{\bG}^{\circ}(x_0)\backslash T_{\bG}(x_0,F')$ because for any $a \in \Stab_{\bG}^{\circ}(x_0)$ and $\Stab_{\bG}^{\circ}(x_0)z \in A_{\bG}(x_0,F')$ we have ${}^zF(a) \in \Stab_{\bG}^{\circ}(x_0)$. We will denote the orbits of this action by $H^1(F,A_{\bG}(x_0,F'))$. Now, if $g\cdot x_0 \in \mathcal{O}^{F'}$ is an $F'$-fixed point then $g^{-1}F(g) \in T_{\bG}(x_0,F')$ and so we have a well defined map $\mathcal{O}^{F'} \to A_{\bG}(x_0,F')$ given by $g\cdot x_0 \mapsto \overline{g^{-1}F(g)} := \Stab_{\bG}^{\circ}(x_0)g^{-1}F(g)$. With this we have the following rephrasing of a classical result, see \cite[4.3.5]{geck:2003:intro-to-algebraic-geometry} for instance.
\end{pa}

\begin{prop}\label{pa:prop-rat-orb-param}
Assume $x_0 \in \mathcal{O}$ then the map $\mathcal{O}^{F'} \to A_{\bG}(x_0,F')$ given by $g\cdot x_0 \mapsto \overline{g^{-1}F(g)}$ induces a bijection $\mathcal{O}^{F'}/\bG^F \to H^1(F,A_{\bG}(x_0,F'))$.
\end{prop}

\begin{proof}
Fix $z \in T_{\bG}(x_0,F')$ and let $g_0 \in \bG$ be such that $g_0^{-1}F(g_0) = z$ then $x_0' = g_0\cdot x_0 \in \mathcal{O}^{F'}$ is $F'$-fixed. It's clear that we have an isomorphism $\imath_{g_0} : A_{\bG}(x_0) \to A_{\bG}(x_0')$ and $F \circ \imath_{g_0} = \imath_{g_0}\circ zF$ so we have a bijection $\imath_{g_0} : H^1(zF,A_{\bG}(x_0)) \to H^1(F,A_{\bG}(x_0'))$. It is easily checked that the bijection $\tau_z : A_{\bG}(x_0) \to A_{\bG}(x_0,F')$ defined by $\tau_z(x) = xz$ induces a bijection $\tau_z : H^1(zF,A_{\bG}(x_0)) \to H^1(F,A_{\bG}(x_0,F'))$. Now consider the sequence of maps
\begin{equation*}
\begin{tikzcd}[row sep=tiny]
\mathcal{O}^{F'}/\bG^F \arrow{r} & \mathcal{O}^{F'}/\bG^F \arrow{r} & H^1(F,A_{\bG}(x_0')) \arrow{r}{\imath_{g_0}^{-1}} & H^1(zF,A_{\bG}(x_0)) \arrow{r}{\tau_z} &  H^1(F,A_{\bG}(x_0,F'))\\
g\cdot x_0  \arrow{r} & gg_0^{-1}\cdot x_0'  \arrow{r} & \overline{g_0g^{-1}F(gg_0^{-1})}  \arrow{r} & \overline{g^{-1}F(g)F(g_0^{-1})g_0} \arrow{r} & \overline{g^{-1}F(g)}
\end{tikzcd}
\end{equation*}
By \cite[4.3.5]{geck:2003:intro-to-algebraic-geometry} the second map is a bijection so we're done.
\end{proof}

\begin{pa}
We will now, additionally, assume that $\phi \in \Aut(\bG,F)$ is a bijective homomorphism of algebraic groups and $\phi' : \mathcal{O} \to \mathcal{O}$ is a bijection such that the following holds:
\begin{enumerate}[label=($\phi$\arabic*)]
	\item $\phi'(g\cdot x) = \phi(g)\cdot\phi'(x)$ for all $g \in \bG$ and $x \in \mathcal{O}$,
	\item $\phi' \circ F' = F'\circ\phi'$.
\end{enumerate}
With these assumptions it is clear that the map $\phi'$ permutes the set $\mathcal{O}^{F'}/\bG^F$ of rational orbits. The following shows how one can understand the action of $\phi'$ on these orbits; the proof is easy and is left as an exercise for the reader.
\end{pa}

\begin{lem}\label{lem:perm-rat-orbits}
Assume $g \in \bG$ is such that $\phi'(x_0) = g\cdot x_0$ and set $z = g^{-1}F(g) \in \Stab_{\bG}(x_0)$. If $\psi : A_{\bG}(x_0,F') \to A_{\bG}(x_0,F')$ is the map induced by $h \mapsto g^{-1}\phi(h)gz$ for all $h \in T_{\bG}(x_0,F')$ then the following diagram is commutative
\begin{equation*}
\begin{tikzcd}
\mathcal{O}^{F'}/G \arrow{d}{}\arrow{r}{\phi'} & \mathcal{O}^{F'}/G \arrow{d}{}\\
H^1(F,A_{\bG}(x_0,F')) \arrow{r}{\psi} & H^1(F,A_{\bG}(x_0,F'))
\end{tikzcd}
\end{equation*}
Here the vertical arrows are given by the bijection in \cref{pa:prop-rat-orb-param}. In particular, if $x_0$ is $\phi'$-fixed then the bijection in \cref{pa:prop-rat-orb-param} is equivariant with respect to $\phi'$ and $\phi$.
\end{lem}


\begin{rem}
We note that even if $\phi'$ stabilises the rational orbit containing $x_0$ it is not necessarily true that $\phi'$ fixes an element in that orbit.
\end{rem}

\section{Isogenies}\label{sec:isogenies}
\begin{pa}
Let $(X,\Phi,\widecheck{X},\widecheck{\Phi})$ be a root datum, c.f., \cite[7.4.1]{springer:2009:linear-algebraic-groups}. In particular, we have $X$ and $\widecheck{X}$ are free $\mathbb{Z}$-modules equipped with a perfect pairing $\langle -,-\rangle : X \times \widecheck{X} \to \mathbb{Z}$. Moreover, $\Phi \subseteq X$, resp., $\widecheck{\Phi} \subseteq \widecheck{X}$, is a root system in the real vector space $X_{\mathbb{R}} := X \otimes_{\mathbb{Z}} \mathbb{R}$, resp., $\widecheck{X}_{\mathbb{R}} := \widecheck{X} \otimes_{\mathbb{Z}} \mathbb{R}$, and we have a bijection $\widecheck{\phantom{\alpha}} : \Phi \to \widecheck{\Phi}$ such that $\langle \alpha,\widecheck{\alpha}\rangle = 2$ for all $\alpha \in \Phi$. Associated to each root $\alpha \in \Phi$ we have a corresponding reflection $s_{\alpha} \in \GL(X_{\mathbb{R}})$ and the subgroup $W \leqslant \GL(X_{\mathbb{R}})$ generated by these reflections is the Weyl group of the root datum. The sextuple $\mathcal{R} = (X,\Phi,\Delta,\widecheck{X},\widecheck{\Phi},\widecheck{\Delta})$ is called a \emph{based root datum} if $\Delta \subseteq \Phi$ is a simple system of roots and $\widecheck{\Delta} \subseteq \widecheck{\Phi}$ is the image of $\Delta$ under the bijection $\Phi \to \widecheck{\Phi}$; note $\widecheck{\Delta}$ is then a simple system of roots in $\widecheck{\Phi}$.
\end{pa}

\begin{pa}\label{pa:iso-def}
Assume now that $\mathcal{R} = (X,\Phi,\Delta,\widecheck{X},\widecheck{\Phi},\widecheck{\Delta})$ and $\mathcal{R}' = (X',\Phi',\Delta',\widecheck{X}',\widecheck{\Phi}',\widecheck{\Delta}')$ are based root data. Recall that if $\varphi \in \Hom(X',X)$ is a $\mathbb{Z}$-module homomorphism then there is a corresponding homomorphism $\widecheck{\varphi} \in \Hom(\widecheck{X},\widecheck{X}')$ called the dual of $\varphi$. Specifically we have $\widecheck{\varphi}$ is the unique homomorphism satisfying $\langle \varphi(x),y\rangle = \langle x,\widecheck{\varphi}(y)\rangle$ for all $x \in X'$ and $y \in \widecheck{X}$. We note that the map $\widecheck{\phantom{\varphi}} : \Hom(X',X) \to \Hom(\widecheck{X},\widecheck{X}')$ is bijective and contravariant. With this we have a $p$-isogeny $\varphi : \mathcal{R}' \to \mathcal{R}$ of the based root data is a homomorphism $\varphi \in \Hom(X',X)$ such that the following hold:
\begin{enumerate}
	\item $\varphi$ and $\widecheck{\varphi}$ are injective
	\item there exists a bijection $b : \Delta \to \Delta'$ and a map $q : \Delta \to \{p^a \mid a \in \mathbb{Z}_{\geqslant 0}\}$ such that for any $\alpha \in \Delta$ we have $\varphi(\beta) = q(\alpha)\alpha$ and $\widecheck{\varphi}(\widecheck{\alpha}) = q(\alpha)\widecheck{\beta}$ where $\beta = b(\alpha)$.
\end{enumerate}
We will denote by $\Iso_p(\mathcal{R}',\mathcal{R})$ the set of all $p$-isogenies $\varphi : \mathcal{R}' \to \mathcal{R}$. Moreover, we set $\Iso_p(\mathcal{R}) = \Iso_p(\mathcal{R},\mathcal{R})$. The following is an easy consequence of the definition.
\end{pa}

\begin{lem}\label{lem:dual-iso}
The map $\widecheck{\phantom{\varphi}} : \Hom(X',X) \to \Hom(\widecheck{X},\widecheck{X}')$ induces a bijection $\Iso_p(\mathcal{R}',\mathcal{R}) \to \Iso_p(\widecheck{\mathcal{R}},\widecheck{\mathcal{R}}')$. 
\end{lem}

\begin{pa}\label{pa:root-datum-alg-grp}
Now assume $\bG$ is a connected reductive algebraic $\mathbb{K}$-group with Borel subgroup $\bB_0 \leqslant \bG$ and maximal torus $\bT_0 \leqslant \bB_0$. We denote by $\mathcal{G}$ the triple $(\bG,\bB_0,\bT_0)$. With respect to $\bT_0$ we have the root datum $(X(\bT_0),\Phi,\widecheck{X}(\bT_0),\widecheck{\Phi})$ of $\bG$. Here $X(\bT_0) = \Hom(\bT_0,\mathbb{G}_m)$ and $\widecheck{X}(\bT_0) = \Hom(\mathbb{G}_m,\bT_0)$ are, respectively, the character and cocharacter groups of $\bT_0$ equipped with the usual perfect pairing defined by $x(y(k)) = k^{\langle x,y\rangle}$ for all $k \in \mathbb{G}_m$, $x \in X(\bT_0)$, and $y \in \widecheck{X}(\bT_0)$. If $\Phi \subseteq X$ are the roots of $\bG$ with respect to $\bT_0$ then associated to each $\alpha \in \Phi$ we have a corresponding $1$-dimensional root subgroup $\bX_{\alpha} \leqslant \bG$. The set of all roots $\alpha \in \Phi$ such that $\bX_{\alpha} \leqslant \bB_0$ is a positive system of roots which contains a unique simple system of roots, say $\Delta$. We then have $\mathcal{R}(\mathcal{G}) = (X(\bT_0),\Phi,\Delta,\widecheck{X}(\bT_0),\widecheck{\Phi},\widecheck{\Delta})$ is a based root datum determined by the triple $\mathcal{G}$.
\end{pa}

\begin{pa}\label{pa:group-isogeny}
Let $\mathcal{G} = (\bG,\bB_0,\bT_0)$ and $\mathcal{G}' = (\bG',\bB_0',\bT_0')$ be triples, as in \cref{pa:root-datum-alg-grp}, with corresponding based root data $\mathcal{R} = \mathcal{R}(\mathcal{G})$ and $\mathcal{R}' = \mathcal{R}(\mathcal{G}')$. Recall that an isogeny $\phi : \bG \to \bG'$ between algebraic groups is a surjective homomorphism of algebraic groups with finite kernel. We will denote by $\Iso(\mathcal{G},\mathcal{G}')$ the set of all isogenies $\phi : \bG \to \bG'$ such that $\phi(\bB_0) = \bB_0'$ and $\phi(\bT_0) = \bT_0'$. Moreover, if $\mathcal{G} = \mathcal{G}'$ then we set $\Iso(\mathcal{G}) := \Iso(\mathcal{G},\mathcal{G})$. Note that we have an action of $\bT_0$ on $\Iso(\mathcal{G},\mathcal{G}')$ given by $t\cdot \phi = \phi \circ \imath_t^{-1}$. Now, if $\phi \in \Iso(\mathcal{G},\mathcal{G}')$ is an isogeny then we have induced maps $X(\phi) : X(\bT_0') \to X(\bT_0)$ and $\widecheck{X}(\phi) : \widecheck{X}(\bT_0) \to \widecheck{X}(\bT_0')$ defined by $X(\phi)(x) = x\circ\phi$ and $\widecheck{X}(\phi)(y) = \phi\circ y$ respectively. These maps are dual to each other, i.e., $\widecheck{X(\phi)} = \widecheck{X}(\phi)$. Moreover, $X(\phi)$ is a $p$-isogeny $\mathcal{R}' \to \mathcal{R}$ of the corresponding based root data and the map $\Iso(\mathcal{G},\mathcal{G}') \to \Iso_p(\mathcal{R}',\mathcal{R})$ defined by $\phi \mapsto X(\phi)$ is constant on $\bT_0$-orbits.
\end{pa}

\begin{pa}
Now assume $F \in \Iso(\mathcal{G})$ and $F' \in \Iso(\mathcal{G}')$ are Steinberg endomorphisms. We will denote by $\Iso((\mathcal{G},F),(\mathcal{G}',F')) \subseteq \Iso(\mathcal{G},\mathcal{G}')$ those isogenies which commute with $F$ and $F'$, i.e., $\phi \circ F = F' \circ \phi$. We also set $\Iso(\mathcal{G},F) = \Iso((\mathcal{G},F),(\mathcal{G},F))$. Moreover, we denote by $\Iso_p((\mathcal{R}',F'),(\mathcal{R},F)) \subseteq \Iso_p(\mathcal{R}',\mathcal{R})$ the set of $p$-isogenies commuting with $X(F)$ and $X(F')$. As before, we set $\Iso_p(\mathcal{R},F) = \Iso_p((\mathcal{R},F),(\mathcal{R},F))$. We will need the following result known as the isogeny theorem, see \cite[1.5]{steinberg:1999:the-isomorphism-and-isogeny-theorems} and \cite[1.4.23]{geck-malle:2016:reductive-groups-and-steinberg-maps}.
\end{pa}

\begin{thm}[Chevalley]\label{thm:isogeny}
The map $\phi \mapsto X(\phi)$ induces a bijection $\Iso(\mathcal{G},\mathcal{G}')/\bT_0 \to \Iso_p(\mathcal{R}',\mathcal{R})$ and a surjection $\Iso((\mathcal{G},F),(\mathcal{G}',F')) \to \Iso_p((\mathcal{R}',F'),(\mathcal{R},F))$.
\end{thm}

\section{Duality}\label{sec:duality}
\begin{pa}\label{pa:notation}
Let us assume that $\mathcal{G} = (\bG,\bB_0,\bT_0)$ and $\mathcal{R} = \mathcal{R}(\mathcal{G}) = (X,\Phi,\Delta,\widecheck{X},\widecheck{\Phi},\widecheck{\Delta})$ are as in \cref{pa:root-datum-alg-grp}. By the existence theorem for reductive groups there exists a triple $\mathcal{G}^{\star} = (\bG^{\star},\bB_0^{\star},\bT_0^{\star})$ with based root datum $\mathcal{R}^{\star} := \mathcal{R}(\mathcal{G}^{\star}) = (X^{\star},\Delta^{\star},\widecheck{X}^{\star},\widecheck{\Delta}^{\star})$ such that we have an isomorphism $\mathcal{R}^{\star} \to \widecheck{\mathcal{R}}$ of based root data. We say the pair $(\mathcal{G},\mathcal{G}^{\star})$ is a dual pair or that $\mathcal{G}^{\star}$ is dual to $\mathcal{G}$. Specifically, if $(\mathcal{G},\mathcal{G}^{\star})$ is a dual pair then we have a $\mathbb{Z}$-module isomorphism $\delta : X^{\star} \to \widecheck{X}$ such that $\delta(\Delta^{\star}) = \widecheck{\Delta}$ and $\widecheck{\delta}(\Delta) = \widecheck{\Delta}^{\star}$ where $\widecheck{\delta} : X \to \widecheck{X}^{\star}$ is the dual of $\delta$. Now assume $(\mathcal{G}',\mathcal{G}'^{\star})$ is another dual pair, so we have an isomorphism $\delta' : \mathcal{R}'^{\star} \to \widecheck{\mathcal{R}}'$ of based root data then we have the following easy lemma.
\end{pa}

\begin{lem}\label{lem:dual-isogenies}
The map ${}^{\star} : \Hom(X',X) \to \Hom(X^{\star},X'^{\star})$ defined by $\varphi^{\star} = \delta'^{-1}\circ \widecheck{\varphi}\circ\delta$ is bijective and contravariant and maps $\Iso_p(\mathcal{R}',\mathcal{R})$ onto $\Iso_p(\mathcal{R}^{\star},\mathcal{R}'^{\star})$.
\end{lem}

\begin{pa}\label{pa:dual-isogenies}
Let us assume $\mathcal{G}' = (\bG',\bB_0',\bT_0')$ and $\mathcal{G}'^{\star} = (\bG'^{\star},\bB_0'^{\star},\bT_0'^{\star})$. As in \cref{pa:group-isogeny} we have an action of $\bT_0'^{\star}$ on $\Iso(\mathcal{G}'^{\star},\mathcal{G}^{\star})$ and combining \cref{lem:dual-isogenies} with \cref{thm:isogeny} we obtain a bijection
\begin{equation}\label{eq:isogeny-bijection}
\Iso(\mathcal{G},\mathcal{G}')/\bT_0 \to \Iso(\mathcal{G}'^{\star},\mathcal{G}^{\star})/\bT_0'^{\star}
\end{equation}
between the orbits. We say an isogeny $\sigma^{\star} \in \Iso(\mathcal{G}'^{\star},\mathcal{G}^{\star})$ is \emph{dual} to $\sigma \in \Iso(\mathcal{G},\mathcal{G}')$ if $X(\sigma^{\star}) = X(\sigma)^{\star}$, i.e., their orbits correspond under the bijection in \cref{eq:isogeny-bijection}. The following is useful to note.
\end{pa}

\begin{lem}\label{lem:injective-dual-isog}
If $\sigma^{\star} \in \Iso(\mathcal{G}'^{\star},\mathcal{G}^{\star})$ is an isogeny dual to $\sigma \in \Iso(\mathcal{G},\mathcal{G}')$ then $\sigma$ is injective if and only if $\sigma^{\star}$ is injective. Moreover, if $\sigma$ is injective then $\sigma^{\star-1}$ is dual to $\sigma^{-1}$.
\end{lem}

\begin{proof}
If $\phi : M \to N$ is a homomorphism of $\mathbb{Z}$-modules then we set $\Coker(\phi) = N/\phi(M)$ and $\Coker(\phi)_{p'}$ to be the quotient of $\Coker(\phi)$ by its $p$-torsion submodule. By \cite[1.11]{bonnafe:2006:sln} we have $\sigma$ is injective if and only if $\Coker(\widecheck{X}(\sigma))_{p'}=0$ and similarly $\sigma^{\star}$ is injective if and only if $\Coker(\widecheck{X}(\sigma^{\star}))_{p'}=0$. From the definition of $\sigma^{\star}$ we easily see that $\widecheck{\delta}$ induces an isomorphism $\Coker(X(\sigma)) \to \Coker(\widecheck{X}(\sigma^{\star}))$. Moreover, by \cite[6.2.3]{sga:2011:sga3-TomeIII}, we have $\Coker(X(\sigma))$ and $\Coker(\widecheck{X}(\sigma))$ are finite groups in duality so $\Coker(\widecheck{X}(\sigma^{\star}))_{p'} \cong \Coker(X(\sigma))_{p'} = 0$ if and only if $\Coker(\widecheck{X}(\sigma))_{p'} = 0$. This proves the first statement.

Now assume $\sigma$ is injective then so is $\sigma^{\star}$. As $\sigma^{\star}$ is dual to $\sigma$ we have $\delta'\circ X(\sigma^{\star}) = \widecheck{X}(\sigma)\circ\delta$ but this implies $\widecheck{X}(\sigma^{-1})\circ\delta' = \delta\circ X(\sigma^{\star-1})$ because $\widecheck{X}(\sigma^{-1})\circ\widecheck{X}(\sigma) = \widecheck{X}(\sigma^{-1}\circ\sigma) = \ID_{\widecheck{X}}$ and $X(\sigma^{\star})\circ X(\sigma^{\star -1}) = X(\sigma^{\star-1}\circ\sigma^{\star}) = \ID_X$. Hence $X(\sigma^{\star-1}) = X(\sigma^{-1})^{\star}$ so $\sigma^{\star-1}$ is dual to $\sigma^{-1}$.
\end{proof}

\begin{pa}
We now consider duality in the presence of Steinberg endomorphisms. For this, assume $F \in \Iso(\mathcal{G})$ and $F' \in \Iso(\mathcal{G}')$ are Steinberg endomorphisms then we assume chosen isogenies $F^{\star} \in \Iso(\mathcal{G}^{\star})$ and $F'^{\star} \in \Iso(\mathcal{G}'^{\star})$ dual to $F$ and $F'$ respectively. The isogenies $F^{\star}$ and $F'^{\star}$ are then also Steinberg endomorphisms and are Frobenius endomorphisms if $F$ and $F'$ are, see \cite[1.4.17, 1.4.27]{geck-malle:2016:reductive-groups-and-steinberg-maps}. If $\sigma \in \Iso((\mathcal{G},F),(\mathcal{G}',F'))$ is an isogeny commuting with $F$ and $F'$ then it follows from \cref{thm:isogeny} that there exists an isogeny $\sigma^{\star} \in \Iso((\mathcal{G}'^{\star},F'^{\star}),(\mathcal{G}^{\star},F^{\star}))$, commuting with $F'^{\star}$ and $F^{\star}$, whose orbit corresponds to that of $\sigma$ under the bijection in \cref{eq:isogeny-bijection}. In this case we say $\sigma^{\star}$ is \emph{dual} to $\sigma$.
\end{pa}

\subsection{Levi Subgroups}
\begin{pa}
For each subset $I \subseteq \Delta$ we have a corresponding parabolic subgroup $\bP_I = \langle \bB_0, \bX_{-\alpha} \mid \alpha \in I\rangle$ of $\bG$ which has a Levi complement $\bL_I = \langle \bT_0, \bX_{\pm\alpha} \mid \alpha \in I\rangle$; this is a connected reductive algebraic group. The root system of $\bL_I$ is the parabolic subsystem $\langle I \rangle \subseteq \Phi$ generated by $I$. The intersection $\bB_I = \bL_I \cap \bB_0$ is a Borel subgroup of $\bL_I$ so we obtain a triple $\mathcal{L}_I = (\bL_I,\bB_I,\bT_0)$ whose associated based root datum is $\mathcal{R}_I = (X,\langle I \rangle,I,\widecheck{X},\langle \widecheck{I}\rangle,\widecheck{I})$, where $\widecheck{I} \subseteq \widecheck{\Delta}$ is the images of $I$ under the bijection $\Delta \to \widecheck{\Delta}$ and $\langle \widecheck{I}\rangle \subseteq \widecheck{\Phi}$ is the parabolic subsystem generated by $\widecheck{I}$. Recall that if $\bP \leqslant \bG$ is any parabolic subgroup with Levi complement $\bL \leqslant \bP$ then there exists a unique subset $I \subseteq \Delta$ such that $\bP = {}^g\bP_I$ and $\bL = {}^g\bL_I$ for some $g \in \bG$.
\end{pa}

\begin{pa}
To $I$ we have a corresponding subset $\widecheck{I}^{\star} := \delta(I) \subseteq \widecheck{\Delta}^{\star}$ of simple coroots for the dual $\mathcal{G}^{\star}$. We then set $I^{\star} \subseteq \Delta^{\star}$ to be the unique set of roots mapping onto $\widecheck{I}^{\star}$ under the bijection $\Delta^{\star} \to \widecheck{\Delta}^{\star}$. As above we obtain a parabolic subgroup $\bP_I^{\star} := \bP_{I^{\star}} \leqslant \bG^{\star}$ with Levi complement $\bL_I^{\star} := \bL_{I^{\star}}$. If $\bB_I^{\star}$ is the intersection $\bL_I^{\star} \cap \bB_0^{\star}$ then $\mathcal{L}_I^{\star} = (\bL_I^{\star}, \bB_I^{\star}, \bT_0^{\star})$ is dual to $\mathcal{L}_I = (\bL_I,\bB_I,\bT_0)$ and the associated based root datum is $\mathcal{R}_I^{\star} = (X^{\star},\langle I^{\star}\rangle,I^{\star},\widecheck{X}^{\star},\langle \widecheck{I}^{\star} \rangle,\widecheck{I}^{\star})$, with the notation as above. The isomorphism between the based root data is simply given by $\delta : X^{\star} \to \widecheck{X}$.
\end{pa}

\subsection{Weyl Groups}
\begin{pa}\label{pa:induced-iso-weyl}
Let us denote by $W$, resp., $W^{\star}$, the Weyl group $W_{\bG}(\bT_0)$, resp., $W_{\bG^{\star}}(\bT_0^{\star})$, of $\bG$, resp., $\bG^{\star}$. The natural map $W \to \GL(X)$, resp., $W^{\star} \to \GL(X^{\star})$, is a faithful representation of the Weyl group $W$, resp., $W^{\star}$. The anti-isomorphism ${}^{\star} : \GL(X) \to \GL(X^{\star})$ defined by $w^{\star} = \delta^{-1} \circ \widecheck{w} \circ \delta$ maps $W$ onto $W^{\star}$. Moreover, assume $\sigma \in \Iso(\mathcal{G},\mathcal{G}')$ and $\sigma^{\star} \in \Iso(\mathcal{G}'^{\star},\mathcal{G}^{\star})$ are dual isogenies then these induce isomorphisms $\sigma : W \to W'$ and $\sigma^{\star} : W'^{\star} \to W^{\star}$, where $W' = W_{\bG'}(\bT_0')$ and $W'^{\star} = W_{\bG'^{\star}}(\bT_0'^{\star})$. With this we have the following diagram is commutative
\begin{equation*}
\begin{tikzcd}
W \arrow{d}{{}^{\star}}\arrow{r}{\sigma} & W' \arrow{d}{{}^{\star}}\\
W^{\star} \arrow{r}{\sigma^{\star -1}} & W'^{\star}
\end{tikzcd}
\end{equation*}
where ${}^{\star} : W' \to W'^{\star}$ is the anti-isomorphism induced by $\delta'$.
\end{pa}

\begin{pa}
If $I \subseteq \Delta$ is a subset then we have a corresponding parabolic subgroup $W_I \leqslant W$ generated by the reflections corresponding to the simple roots in $I$; this is the Weyl group $W_{\bL_I}(\bT_0)$ of the Levi subgroup $\bL_I$ defined with respect to $\bT_0$. The anti-isomorphism ${}^{\star}$ maps the parabolic subgroup $W_I$ onto the parabolic subgroup of $W^{\star}$ generated by the reflections corresponding to the roots in $I^{\star}$.
\end{pa}

\section{Duality Between Tori and Lusztig Series}\label{sec:duality-Lusztig-series}
\begin{assumption}
From this point forward we assume that $\mathcal{G} = (\bG,\bB_0,\bT_0)$ and $\mathcal{G}^{\star} = (\bG^{\star},\bB_0^{\star},\bT_0^{\star})$ are dual triples, as in \cref{pa:notation}, and $F \in \Iso(\mathcal{G})$ and $F^{\star} \in \Iso(\mathcal{G}^{\star})$ are dual Steinberg endomorphisms. Moreover, for any element $w \in W = W_{\bG}(\bT_0)$, resp., $x \in W^{\star} = W_{\bG^{\star}}(\bT_0^{\star})$, we assume chosen an element $n_w \in N_{\bG}(\bT_0)$, resp., $n_x \in N_{\bG^{\star}}(\bT_0^{\star})$, representing $w$, resp., $x$.
\end{assumption}

\begin{pa}\label{pa:Steinberg-end}
For any $w \in W$ we set $Fw := Fn_w : \bG \to \bG$. This is a Steinberg endomorphism of $\bG$ stabilising $\bT_0$ so for any $w \in W$ we obtain a finite subgroup $\bT_0^{Fw} \leqslant \bT_0$. We note that the restriction of $Fw$ to $\bT_0$ does not depend upon the choice of representative $n_w$ used to define it. We will denote by $\mathcal{C}_{\bG}(\bT_0,F)$ the set of all pairs $(w,\theta)$ consisting of an element $w \in W$ and an irreducible character $\theta \in \Irr(\bT_0^{Fw})$. The group $W$ acts on $\mathcal{C}_{\bG}(\bT_0,F)$ via $z\cdot (w,\theta) = (zwF(z)^{-1},{}^{F(n_z)}\theta)$. Now let us denote by $\mathcal{C}(\bG,F)$ the set of all pairs $(\bT,\theta)$ consisting of an $F$-stable maximal torus $\bT \leqslant \bG$ and an irreducible character $\theta \in \Irr(\bT^F)$. Clearly the group $\bG^F$ acts on $\mathcal{C}(\bG,F)$ via $g\cdot (\bT,\theta) = ({}^g\bT,{}^g\theta)$. Given $(w,\theta) \in \mathcal{C}_{\bG}(\bT_0,F)$ we obtain an element $({}^{g_w}\bT_0,{}^{g_w}\theta) \in \mathcal{C}(\bG,F)$ by choosing an element $g_w \in \bG$ such that $g_w^{-1}F(g_w) = F(n_w)$. We then have the following well known lemma.
\end{pa}

\begin{lem}\label{lem:char-orb-bij}
The map $\mathcal{C}_{\bG}(\bT_0,F)/W \to \mathcal{C}(\bG,F)/\bG^F$ defined by $[w,\theta] \mapsto [{}^{g_w}\bT_0,{}^{g_w}\theta]$ is a well-defined bijection.
\end{lem}

\begin{pa}\label{pa:dual-params}
As above, for any $w \in W^{\star}$ we set $wF^{\star} := n_wF^{\star} : \bG^{\star} \to \bG^{\star}$. Again, this is a Steinberg endomorphism of $\bG$ stabilising $\bT_0^{\star}$ so for any $w \in W^{\star}$ we obtain a finite subgroup $\bT_0^{\star wF^{\star}} \leqslant \bT_0^{\star}$. We will denote by $\mathcal{S}_{\bG^{\star}}(\bT_0^{\star},F^{\star})$ the set of all pairs $(w,s)$ consisting of an element $w \in W^{\star}$ and a semisimple element $s \in \bT_0^{\star wF^{\star}}$. The group $W^{\star}$ acts on this set via $z\cdot (w,s) = (zwF^{\star}(z)^{-1},{}^{n_z}s)$. Similarly to before we denote by $\mathcal{S}(\bG^{\star},F^{\star})$ the set of all pairs $(\bT^{\star},s)$ consisting of an $F^{\star}$-stable maximal torus $\bT^{\star} \leqslant \bG^{\star}$ and a semisimple element $s \in \bT^{\star F^{\star}}$. The finite group $\bG^{\star F^{\star}}$ acts on $\mathcal{S}(\bG^{\star},F^{\star})$ via $g\cdot (\bT^{\star},s) = ({}^g\bT^{\star},{}^gs)$. Given $(w,s) \in \mathcal{S}_{\bG^{\star}}(\bT_0^{\star},F^{\star})$ we obtain an element $({}^{g_w}\bT_0^{\star},{}^{g_w}s) \in \mathcal{S}(\bG^{\star},F^{\star})$ by choosing an element $g_w \in \bG^{\star}$ such that $g_w^{-1}F^{\star}(g_w) = n_w$. Again, we then have the following well known lemma.
\end{pa}

\begin{lem}\label{lem:s/s-orb-bij}
The map $\mathcal{S}_{\bG^{\star}}(\bT_0^{\star},F^{\star})/W^{\star} \to \mathcal{S}(\bG^{\star},F^{\star})/\bG^{\star F^{\star}}$ defined by $[w,s] \mapsto [{}^{g_w}\bT_0^{\star},{}^{g_w}s]$ is a well-defined bijection.
\end{lem}

\subsection{Lusztig Series}
\begin{pa}\label{pa:char-to-elms-iso}
For any $w \in W$ we have the Frobenius endomorphism $Fw$ of $\bT_0$ is dual to the Frobenius endomorphism $w^{\star}F^{\star}$ of $\bT_0^{\star}$ with respect to the isomorphism $\widecheck{\delta} : X \to \widecheck{X}$. Let us assume chosen once and for all an isomorphism $\imath : (\mathbb{Q/Z})_{p'} \to \mathbb{K}^{\times}$ and an embedding $\jmath : \mathbb{K}^{\times} \hookrightarrow \Ql$. As in \cite[13.11]{digne-michel:1991:representations-of-finite-groups-of-lie-type} we may define for any $w \in W$ a group isomorphism $\widecheck{\delta}_w : \Irr(\bT_0^{Fw}) \to \bT_0^{\star w^{\star}F^{\star}}$ as follows. Given $\theta \in \Irr(\bT_0^{Fw})$ we choose an element $\chi \in X$ such that $\jmath\circ\Res^{\bT_0}_{\bT_0^{Fw}}(\chi) = \theta$. Let $n > 0$ be such that $X((w^{\star}F^{\star})^n) = p^k$, with $k > 0$, then we set $\widecheck{\delta}_w(\theta) = N_{w^{\star}F^{\star},n}(\widecheck{\delta}(\chi)(\zeta))$ where $\zeta = \imath(1/(p^k-1))$ and $N_{w^{\star}F^{\star},n} : \bT_0^{\star(w^{\star}F^{\star})^n} \to \bT_0^{\star w^{\star}F^{\star}}$ is the norm map, see \cite[11.9]{digne-michel:1991:representations-of-finite-groups-of-lie-type}. With this, the following lemma is easy.
\end{pa}

\begin{lem}\label{lem:bij-DL-chars}
The map $\mathcal{C}_{\bG}(\bT_0,F)/W \to \mathcal{S}_{\bG^{\star}}(\bT_0^{\star},F^{\star})/W^{\star}$ defined by $[w,\theta] \mapsto [w^{\star},\widecheck{\delta}_w(\theta)]$ is a bijection.
\end{lem}

\begin{cor}\label{cor:nabla-bijection}
We have a well-defined bijection
\begin{equation*}
\mathcal{C}(\bG,F)/\bG^F \to \mathcal{C}_{\bG}(\bT_0,F)/W \to \mathcal{S}_{\bG^{\star}}(\bT_0^{\star},F^{\star})/W^{\star} \to \mathcal{S}(\bG^{\star},F^{\star})/\bG^{\star F^{\star}}
\end{equation*}
given by composing the bijections in \cref{lem:char-orb-bij,lem:s/s-orb-bij,lem:bij-DL-chars}.
\end{cor}

\begin{pa}
Now, for any parabolic subgroup $\bP \leqslant \bG$ with $F$-stable Levi complement $\bL \leqslant \bP$ we have a Deligne--Lusztig induction map $R_{\bL \subseteq \bP}^{\bG} : \Irr(\bL^F) \to \mathbb{Z}\Irr(\bG^F)$ defined as follows. Let $\bU \leqslant \bP$ be the unipotent radical of the parabolic then we consider the variety $\bY_{\bU}^{\bG} = \{x \in \bG \mid x^{-1}F(x) \in F(\bU)\}$. The group $\bG^F \times (\bL^F)^{\opp}$ acts on $\bY_{\bU}^{\bG}$ as a finite group of automorphisms via $(g,l)\cdot x = gxl$. By the functoriality of $\ell$-adic cohomology with respect to finite morphisms this endows each compactly supported $\ell$-adic cohomology group $H_c^i(\bY_{\bU}^{\bG})$ with the structure of a $(\Ql\bG^F,\Ql\bL^F)$-bimodule. We then have
\begin{equation*}
R_{\bL \subseteq \bP}^{\bG}(\chi) = \sum_{i \in \mathbb{Z}} (-1)^i\Tr(-,H_c^i(\bY_{\bU}^{\bG}) \otimes_{\Ql \bL^F} M_{\chi}),
\end{equation*}
where $M_{\chi}$ is a left $\Ql \bL^F$-module affording the character $\chi$ and $H_c^i(\bY_{\bU}^{\bG})$ is the $i$th compactly supported $\ell$-adic cohomology of the variety.
\end{pa}

\begin{pa}
With this we may define (rational) Lusztig series. For each pair $(\bT,\theta) \in \mathcal{C}(\bG,F)$ we have a corresponding Deligne--Lusztig virtual character $R_{\bT}^{\bG}(\theta) := R_{\bT \subseteq \bB}^{\bG}(\theta)$ where $\bB$ is some Borel subgroup containing $\bT$ (in this case $R_{\bT \subseteq \bB}^{\bG}$ does not depend upon the choice of $\bB$). If $[\bT^{\star},s] \in \mathcal{S}(\bG^{\star},F^{\star})/\bG^{\star F^{\star}}$ corresponds to $[\bT,\theta] \in \mathcal{C}(\bG,F)/\bG^F$ under the previous bijection then we denote the virtual character $R_{\bT}^{\bG}(\theta)$ by $R_{\bT^{\star}}^{\bG}(s)$. Let $s \in \bG^{\star F^{\star}}$ be a semisimple element then we define $\mathcal{E}(\bG^F,s)$ to be the set of all irreducible characters occurring in a Deligne--Lusztig virtual character $R_{\bS^{\star}}^{\bG}(t)$ with $t$ a semisimple element $\bG^{\star F^{\star}}$-conjugate to $s$.
\end{pa}

\section{Lusztig Series and Isogenies}
\begin{pa}
Let us assume that $\mathcal{G}' = (\bG',\bB_0',\bT_0')$ and $\mathcal{G}'^{\star} = (\bG'^{\star},\bB_0'^{\star},\bT_0'^{\star})$ are another set of dual triples endowed with dual Steinberg endomorphisms $F' \in \Iso(\bG')$ and $F'^{\star} \in \Iso(\bG'^{\star})$. Moreover, we assume $\sigma \in \Iso((\mathcal{G},F),(\mathcal{G}',F'))$ is an injective isogeny, i.e., a bijective homomorphism, then $\sigma$ restricts to an isomorphism of finite groups $\bG^F \to \bG'^{F'}$. In addition, we assume $\sigma^{\star} \in \Iso((\mathcal{G}'^{\star},F'^{\star}),(\mathcal{G}^{\star},F^{\star}))$ is an isogeny dual to $\sigma$ then this is also injective and induces an isomorphism of finite groups $\bG'^{\star F'^{\star}} \to \bG^{\star F^{\star}}$, c.f., \cref{lem:injective-dual-isog}. The following shows what happens to Lusztig series when identifying the irreducible characters of $\bG^F$ and $\bG'^{F'}$ through $\sigma$.
\end{pa}

\begin{prop}\label{prop:Lusztig-series-iso-image}
Assume $s \in \bG^{\star F^{\star}}$ is a semisimple element then ${}^{\sigma}\mathcal{E}(\bG^F,s) = \mathcal{E}(\bG'^{F'},\sigma^{\star-1}(s))$.
\end{prop}

\begin{proof}
Let $(\bT,\theta) \in \mathcal{C}(\bG,F)$ then by \cite[13.22]{digne-michel:1991:representations-of-finite-groups-of-lie-type} we have for any $(\bT,\theta) \in \mathcal{C}(\bG,F)$ that ${}^{\sigma}R_{\bT}^{\bG}(\theta) = R_{\sigma(\bT)}^{\bG'}({}^{\sigma}\theta)$. If $[\bT,\theta]$ corresponds to $[\bT^{\star},s] \in \mathcal{S}(\bG^{\star},F^{\star})$ under the bijection in \cref{cor:nabla-bijection} then the statement follows if we can show that $[\sigma(\bT),{}^{\sigma}\theta] \in \mathcal{C}(\bG',F')$ corresponds to $[\sigma^{\star-1}(\bT^{\star}),\sigma^{\star-1}(s)] \in \mathcal{S}(\bG'^{\star},F'^{\star})$.

Let us assume that $[\bT,\theta]$ corresponds to $[w,\theta_0] \in \mathcal{C}_{\bG}(\bT_0,F)/W$ under the map in \cref{lem:char-orb-bij} then by definition we have $[\bT^{\star},s]$ corresponds to $[w^{\star},s_0] \in \mathcal{S}_{\bG^{\star}}(\bT_0^{\star},F^{\star})/W^{\star}$ under the map in \cref{lem:s/s-orb-bij} where $s_0 = \widecheck{\delta}_w(\theta_0)$. It's easy to check that $[\sigma(\bT),{}^{\sigma}\theta]$ corresponds to $[\sigma(w),{}^{\sigma}\theta_0]$ under the map in \cref{lem:char-orb-bij} and $[\sigma^{\star-1}(\bT^{\star}),\sigma^{\star-1}(s)]$ corresponds to $[\sigma^{\star-1}(w^{\star}),\sigma^{\star-1}(s_0)]$ under the map in \cref{lem:s/s-orb-bij}. Hence, we need only show that $[\sigma(w),{}^{\sigma}\theta_0]$ is mapped onto $[\sigma^{\star-1}(w^{\star}),\sigma^{\star-1}(s_0)]$ under the map in \cref{lem:bij-DL-chars}.

By definition we have $[\sigma(w),{}^{\sigma}\theta_0]$ is mapped onto $[\sigma(w)^{\star},\widecheck{\delta}_w'({}^{\sigma}\theta_0)]$ and $\sigma(w)^{\star} = \sigma^{\star-1}(w^{\star})$, c.f., \cref{pa:induced-iso-weyl}. Thus, it suffices to show that $\widecheck{\delta}_w'({}^{\sigma}\theta_0) = \sigma^{\star-1}(s_0)$. For this, assume $\chi \in X$ satisfies $\jmath\circ\Res_{\bT_0^{Fw}}^{\bT_0}(\chi) = \theta$ then $\Res_{\bT_0'^{F'\sigma(w)}}^{\bT_0'}({}^{\sigma}\chi) = {}^{\sigma}\theta$. Rewriting we have ${}^{\sigma}\chi=X(\sigma^{-1})(\chi)$ and by \cref{lem:injective-dual-isog} we have $\widecheck{\delta}'\circ X(\sigma^{-1}) = \widecheck{X}(\sigma^{\star-1})\circ \widecheck{\delta}$. From this the statement follows easily from the description of the map $\widecheck{\delta}_w'$, c.f., \cref{pa:char-to-elms-iso}, because there is a common $n>0$ such that $X((w^{\star}F^{\star})^n) = p^k = X'((\sigma^{\star -1}(w^{\star})F'^{\star})^n)$ and $N_{\sigma^{\star -1}(w^{\star})F'^{\star},n} \circ \widecheck{X}(\sigma^{\star-1}) = \widecheck{X}(\sigma^{\star-1})\circ N_{w^{\star}F^{\star},n}$.
\end{proof}

\begin{rem}
In \cite[Corollary 2.4]{navarro-tiep-turull:2008:brauer-characters-with-cyclotomic} it is stated that if $\sigma : \bG \to \bG$ is a bijective endomorphism of $\bG$ commuting with $F$ then ${}^{\sigma}\mathcal{E}(\bG^F,s) = \mathcal{E}(\bG^F,\sigma^{\star}(s))$. Unfortunately, the definition of the dual $\sigma^{\star}$ is not explicitly given in \cite{navarro-tiep-turull:2008:brauer-characters-with-cyclotomic} and the main part of the proof of this statement is left to the reader so it is difficult to reconcile that statement with \cref{prop:Lusztig-series-iso-image}. To avoid any confusion we have decided to give a complete proof of \cref{prop:Lusztig-series-iso-image}.
\end{rem}

\section{Twisted Induction and Lusztig Series}
\begin{pa}\label{pa:twisted-induction}
We will now rephrase the usual notions of Deligne--Lusztig induction and Lusztig series in a way that is suited to our purpose. Our setup here is similar to that considered in \cite[6.20]{lusztig:1984:characters-of-reductive-groups} and \cite{digne-michel:1990:lusztigs-parametrization}. For this, assume $I \subseteq \Delta$ is a subset of simple roots and let $w \in Z_W(\bL_I,F)$. The corresponding Steinberg endomorphism $Fw$ defined in \cref{pa:Steinberg-end} then stabilises the Levi subgroup $\bL_I$. Consider the variety $\bY_{I,w}^{\bG} = \{x \in \bG \mid x^{-1}F(x) \in F(n_w\bU_I)\}$, where $\bU_I \leqslant \bP_I$ is the unipotent radical of the corresponding parabolic. Following \cite[6.21]{lusztig:1984:characters-of-reductive-groups} we define a map $R_{I,w}^{\bG} : \Irr(\bL_I^{Fw}) \to \mathbb{Z}\Irr(\bG^F)$ by setting
\begin{equation*}
R_{I,w}^{\bG}(\chi) = \sum_{i \in \mathbb{Z}} (-1)^i\Tr( - , H_c^i(\bY_{I,w}^{\bG}) \otimes_{\Ql\bL_I^{F_w}} M_{\chi}),
\end{equation*}
where $M_{\chi}$ denotes a left $\Ql\bL_I^{Fw}$-module affording the character $\chi$. We refer to $R_{I,w}^{\bG}$ as \emph{twisted induction}.
\end{pa}

\begin{pa}\label{pa:twisted-ind-vs-DL-ind}
Assume now $g \in \bG$ is an element such that $g^{-1}F(g) = F(n_w)$ and set $\bP = {}^g\bP_I$ and $\bL = {}^g\bL_I$. Clearly we have $F \circ\imath_g = \imath_g \circ Fw$ and so $\imath_g$ induces an isomorphism $\bL_I^{Fw} \to \bL^F$ of finite groups. The following relates $R_{\bL \subseteq \bP}^{\bG}$ and $R_{I,w}^{\bG}$, where $R_{\bL \subseteq \bP}^{\bG}$ is the Deligne--Lusztig induction map corresponding to $\bL \leqslant \bP$. We note the proof is similar to that of \cite[2.1]{navarro-tiep-turull:2008:brauer-characters-with-cyclotomic}.
\end{pa}

\begin{lem}\label{lem:twisted-ind-DL-ind}
We have $R_{I,w}^{\bG}(\chi) = R_{\bL \subseteq \bP}^{\bG}({}^g\chi)$ for any $\chi \in \Irr(\bL_I^{Fw})$.
\end{lem}

\begin{proof}
If $\bU \leqslant \bP$ is the unipotent radical of the parabolic then clearly $\bU = {}^g\bU_I$. An easy calculation shows that the isomorphism of varieties $\phi : \bG \to \bG$, defined by $\phi(x) = xg$, maps $\bY_{\bU}^{\bG}$ isomorphically onto $\bY_{I,w}^{\bG}$. Thus $\phi$ induces an isomorphism of vector spaces $H_c^i(\bY_{\bU}^{\bG}) \to H_c^i(\bY_{I,w}^{\bG})$ for each $i \in \mathbb{Z}$. As $\phi(hxl) = h\phi(x)l^g$ for any $h \in \bG^F$, $x \in \bY_{\bU}^{\bG}$, and $l \in \bL^F$, we have $\phi$ induces an isomorphism $H_c^i(\bY_{\bU}^{\bG}) \to H_c^i(\bY_{I,w}^{\bG})^g$ of $(\Ql\bG^F,\Ql\bL^F)$-bimodules. Hence, we obtain an isomorphism of left $\Ql\bG^F$-modules
\begin{equation*}
H_c^i(\bY_{\bU}^{\bG}) \otimes_{\Ql\bL^F} M_{\chi} \cong H_c^i(\bY_{I,w}^{\bG})^g \otimes_{\Ql\bL^F} M_{\chi} \cong H_c^i(\bY_{I,w}^{\bG}) \otimes_{\Ql\bL_I^{Fw}} {}^gM_{\chi}.
\end{equation*}
The statement now follows from the fact that ${}^gM_{\chi} \cong M_{{}^g\chi}$.
\end{proof}

\subsection{Rational Lusztig Series}
\begin{pa}\label{pa:DL-char-identification}
We wish to now define Lusztig series in an alternative way. For this, we will need some alternative notation for Deligne--Lusztig virtual characters. If $w \in W$ and $\theta \in \Irr(\bT_0^{Fw})$ then $R_w^{\bG}(\theta) := R_{\emptyset, w}^{\bG}(\theta)$ is nothing other than a Deligne--Lusztig virtual character. More precisely, if $[\bT,\theta'] \in \mathcal{C}(\bG,F)/\bG^F$ corresponds to $[w,\theta]$ under the bijection in \cref{lem:char-orb-bij} then $R_{\bT}^{\bG}(\theta') = R_w^{\bG}(\theta)$ by \cref{lem:twisted-ind-DL-ind}. Now, if $[w^{\star},s] \in \mathcal{S}_{\bG^{\star}}(\bT_0^{\star},F^{\star})/W^{\star}$ corresponds to $[w,\theta] \in \mathcal{C}_{\bG}(\bT_0,F)/W$ under the bijection in \cref{lem:bij-DL-chars} then we will denote by $R_{w^{\star}}^{\bG}(s)$ the virtual character $R_w^{\bG}(\theta)$. Note that when we write $R_{w^{\star}}^{\bG}(s)$ we \emph{implicitly assume} that $w^{\star} \in T_{W^{\star}}(s,F^{\star})$.
\end{pa}

\begin{pa}\label{pa:coset-centraliser}
Let us now fix a semisimple element $s \in \bT_0^{\star}$ and let $\mathcal{O}$ be the $\bG^{\star}$-conjugacy class of $s$. We will aditionally assume that $T_{W^{\star}}(s,F^{\star}) \neq\emptyset$ or equivalently that $F^{\star}(\mathcal{O}) = \mathcal{O}$. By \cref{pa:prop-rat-orb-param} we have a bijection
\begin{equation*}
\mathcal{O}^{F^{\star}}/\bG^{F^{\star}} \to H^1(F^{\star},A_{\bG^{\star}}(s,F^{\star}))
\end{equation*}
Now, the centraliser $W^{\star}(s) = C_{N_{\bG^{\star}}(\bT_0^{\star})}(s)/\bT_0^{\star} \leqslant W^{\star}$ of $s$ in $W^{\star}$ contains the Weyl group $W^{\star\circ}(s) = N_{C_{\bG^{\star}}^{\circ}(s)}(\bT_0^{\star})/\bT_0^{\star}$ of $C_{\bG^{\star}}^{\circ}(s)$ as a normal subgroup and we will denote by $\mathcal{A}_{W^{\star}}(s,F^{\star})$ the set of cosets $W^{\star\circ}(s) \backslash T_{W^{\star}}(s,F^{\star})$. The group $W^{\star}(s)$ acts on $T_{W^{\star}}(s,F^{\star})$ by $F^{\star}$-conjugation and this induces an action of $\mathcal{A}_{W^{\star}}(s) := W^{\star}(s)/W^{\star\circ}(s)$ on $\mathcal{A}_{W^{\star}}(s,F^{\star})$; we denote the resulting set of orbits by $H^1(F^{\star},\mathcal{A}_{W^{\star}}(s,F^{\star}))$. A standard argument shows that the map $\mathcal{A}_{W^{\star}}(s,F^{\star}) \to A_{\bG^{\star}}(s,F^{\star})$ defined by $W^{\star\circ}(s)w \mapsto C_{\bG^{\star}}^{\circ}(s)n_w$ is a bijection and, moreover, this induces a bijection $H^1(F^{\star},\mathcal{A}_{W^{\star}}(s,F^{\star})) \to H^1(F^{\star},A_{\bG^{\star}}(s,F^{\star}))$. With this we define for any $a^{\star} \in \mathcal{A}_{W^{\star}}(s,F^{\star})$ a set
\begin{equation*}
\mathcal{E}_0(\bG^F,s,a^{\star}) = \{\chi \in \Irr(\bG^F) \mid \langle \chi, R_{w^{\star}}^{\bG}(s)\rangle_{\bG^F} \neq 0 \text{ and }W^{\star\circ}(s)w^{\star} \sim_{F^{\star}} a^{\star}\},
\end{equation*}
where $\sim_{F^{\star}}$ denotes the equivalence relation induced by the action of $\mathcal{A}_{W^{\star}}(s)$ via $F^{\star}$-conjugacy. The following shows that this set is a Lusztig series.
\end{pa}

\begin{lem}\label{lem:DL-geo-series-compare}
Assume $t = {}^gs \in \bG^{\star F^{\star}}$ is such that $\overline{g^{-1}F^{\star}(g)} \in A_{\bG^{\star}}(s,F^{\star})$ corresponds to $a^{\star} \in \mathcal{A}_{W^{\star}}(s,F^{\star})$ then $\mathcal{E}(\bG^F,t) = \mathcal{E}_0(\bG^F,s,a^{\star})$.
\end{lem}

\begin{proof}
Let $\mathcal{X}$ denote the set of all pairs $(\bT^{\star},t')$ consisting of a maximal torus $\bT^{\star} \leqslant \bG^{\star}$ and a semisimple element $t' \in \bT^{\star}$ such that $t'$ is $\bG^{\star}$-conjugate to $s$. The set $\mathcal{X}$ is a $\bG^{\star}$-set via the action $g\cdot(\bT^{\star},t') = ({}^g\bT^{\star},{}^gt')$ and we have a map $F' : \mathcal{X} \to \mathcal{X}$ defined by $F'(\bT^{\star},t') = (F^{\star}(\bT^{\star}),F^{\star}(t'))$. This action is transitive, so we're in the situation of \cref{pa:setup}. Clearly $x_0 = (\bT_0^{\star},s) \in \mathcal{X}$ and so we have a bijection $\mathcal{X}^{F'}/\bG^F \to H^1(F^{\star},A_{\bG^{\star}}(x_0,F^{\star}))$ as in \cref{pa:prop-rat-orb-param}. The stabiliser of $x_0$ is $C_{N_{\bG^{\star}}(\bT_0^{\star})}(s)$ whose connected component is $\bT_0^{\star}$. As $T_{\bG^{\star}}(x_0,F^{\star}) = T_{N_{\bG^{\star}}(\bT_0^{\star})}(x_0,F^{\star})$ we must have $A_{\bG^{\star}}(x_0,F^{\star}) = T_{W^{\star}}(s,F^{\star})$. As before this bijection is compatible with the $W^{\star}(s)$-action, resp., $A_{\bG^{\star}}(x_0)$-action, by $F^{\star}$-conjugation. Hence we obtain a bijection $H^1(F^{\star},T_{W^{\star}}(s,F^{\star})) \to H^1(F^{\star},A_{\bG^{\star}}(s,F^{\star}))$. With this one readily checks we have a commutative diagram
\begin{equation*}
\begin{tikzcd}
\mathcal{X}^{F'}/\bG^F \arrow[two heads]{d}{\alpha}\arrow{r}{} & H^1(F^{\star},A_{\bG^{\star}}(x_0,F^{\star})) \arrow{r}{}\arrow[two heads]{d}{\beta} & H^1(F^{\star},T_{W^{\star}}(s,F^{\star})) \arrow[two heads]{d}{\gamma}\\
\mathcal{O}^F/\bG^F \arrow{r}{} & \arrow{r}{} H^1(F^{\star},A_{\bG^{\star}}(s,F^{\star})) & H^1(F^{\star},\mathcal{A}_{W^{\star}}(s,F^{\star}))
\end{tikzcd}
\end{equation*}
where $\alpha(\bT^{\star},t') = t'$, $\beta(\bT_0^{\star}n) = C_{\bG^{\star}}^{\circ}(s)n$, and $\gamma(w^{\star}) = W^{\star\circ}(s)w^{\star}$. The statement now follows from the definitions.
\end{proof}

\begin{pa}\label{pa:disjointness}
Let us denote by $\mathcal{T}_{\bG^{\star}}(\bT_0^{\star},F^{\star})$ the set of all pairs $(s,a^{\star})$ such that $s \in \bT_0^{\star}$, $T_{W^{\star}}(s,F^{\star}) \neq \emptyset$, and $a^{\star} \in \mathcal{A}_{W^{\star}}(s,F^{\star})$. The group $W^{\star}$ acts on this set in the following way. Consider the conjugate $s' = {}^{n_{x^{\star}}}s$ with $x^{\star} \in W^{\star}$ then we have $C_{\bG^{\star}}(s') = {}^{n_{x^{\star}}}C_{\bG^{\star}}(s)$ and $C_{\bG^{\star}}^{\circ}(s') = {}^{n_{x^{\star}}}C_{\bG^{\star}}^{\circ}(s)$. In particular, we must have $W^{\star}(s') = {}^{x^{\star}}W^{\star}(s)$ and $W^{\star\circ}(s') = {}^{x^{\star}}W^{\star\circ}(s)$. The map $\phi : T_{W^{\star}}(s,F^{\star}) \to T_{W^{\star}}(s',F^{\star})$ defined by $\phi(w^{\star}) = x^{\star}w^{\star}F^{\star}(x^{\star-1})$ is a bijection which induces a well-defined bijection $\mathcal{A}_{W^{\star}}(s,F^{\star}) \to \mathcal{A}_{W^{\star}}(s',F^{\star})$ because $W^{\star\circ}(s') = {}^{x^{\star}}W^{\star\circ}(s)$. Moreover, this map is compatible with the actions of $\mathcal{A}_{W^{\star}}(s)$ and $\mathcal{A}_{W^{\star}}(s')$ by $F^{\star}$-conjugacy in the sense that $\phi(y^{\star}\cdot w^{\star}) = ({}^{x^{\star}}y^{\star}) \cdot \phi(w^{\star})$. With this we define a $W^{\star}$-action on $\mathcal{T}_{\bG^{\star}}(\bT_0^{\star},F^{\star})$ by setting
\begin{equation*}
x^{\star} \cdot (s,a^{\star}) = ({}^{n_{x^{\star}}}s,x^{\star}a^{\star}F^{\star}(x^{\star-1})).
\end{equation*}
From the proof of \cref{lem:DL-geo-series-compare}, together with the usual disjointness statment for Lusztig series, we thus conclude that we have a disjoint union
\begin{equation*}
\Irr(\bG^F) = \bigsqcup_{[s,a^{\star}] \in \mathcal{T}_{\bG^{\star}}(\bT_0^{\star},F^{\star})/W^{\star}} \mathcal{E}_0(\bG^F,s,a^{\star}).
\end{equation*}
\end{pa}

\begin{rem}
We note that the orbits $\mathcal{T}_{\bG^{\star}}(\bT_0^{\star},F^{\star})/W^{\star}$ are in bijection with the $\bG^{\star F^{\star}}$-conjugacy classes of rational semisimple elements.
\end{rem}

\subsection{Geometric Series, Cells, and Families}
\begin{pa}\label{pa:geometric-series}
For any $s \in \bT_0^{\star}$ we define a set
\begin{equation*}
\mathcal{E}_0(\bG^F,s) = \{\chi \in \Irr(\bG^F) \mid \langle \chi, R_{w^{\star}}^{\bG}(s)\rangle_{\bG^F} \neq 0 \text{ and }w^{\star} \in T_{W^{\star}}(s,F^{\star})\}.
\end{equation*}
Note that $\mathcal{E}_0(\bG^F,s) = \emptyset$ unless $T_{W^{\star}}(s,F^{\star}) \neq \emptyset$ and moreover we have
\begin{equation*}
\mathcal{E}_0(\bG^F,s) = \bigsqcup_{a^{\star} \in H^1(F^{\star},\mathcal{A}_{W^{\star}}(s,F^{\star}))}\mathcal{E}_0(\bG^F,s,a^{\star}).
\end{equation*}
The set $\mathcal{E}_0(\bG^F,s)$ is a \emph{geometric} Lusztig series. If $W^{\star}(s) = W^{\star\circ}(s)$ then $\mathcal{A}_{W^{\star}}(s,F^{\star})$ has cardinality $1$ and the geometric Lusztig series is a rational Lusztig series. In \cite[16.4]{lusztig:1985:character-sheaves} Lusztig has decomposed the group $W^{\star}(s)$ as a disjoint union of two-sided Kazhdan--Lusztig cells. Moreover, to each such cell $\mathfrak{C} \subseteq W^{\star}(s)$ we have a corresponding family of irreducible characters $\Irr(W^{\star}(s) \mid \mathfrak{C}) \subseteq \Irr(W^{\star}(s))$ and this yields a disjoint union
\begin{equation*}
\Irr(W^{\star}(s)) = \bigsqcup_{\mathfrak{C} \subseteq W^{\star}(s)}\Irr(W^{\star}(s) \mid \mathfrak{C}).
\end{equation*}
We note that each family contains a unique special representation $E_{\mathfrak{C}} \in \Irr(W^{\star}(s) \mid \mathfrak{C})$.
\end{pa}

\begin{pa}
Let us fix an element $w_1^{\star} \in T_{W^{\star}}(s,F^{\star})$ then the map $w_1^{\star}F^{\star}$ defines an automorphism of $W^{\star}(s)$ which permutes the two-sided Kazhdan--Lusztig cells and thus the families. We denote by $\widetilde{W}^{\star}(s)$ the semidirect product $W^{\star}(s) \rtimes \langle w_1^{\star}F^{\star}\rangle$. If $E \in \Irr(W^{\star}(s))^{w_1^{\star}F^{\star}}$ is a $w_1^{\star}F^{\star}$-invariant irreducible character then we denote by $\widetilde{E} \in \Irr(\widetilde{W}^{\star}(s))$ an extension of $E$; this exists because the quotient $\widetilde{W}^{\star}(s)/W^{\star}(s)$ is cyclic. Associated to $\widetilde{E}$ we have a corresponding $\bG^F$-class function
\begin{equation*}
\mathcal{R}^{\bG}_{\bT_0^{\star}}(\widetilde{E},s) = \frac{1}{|W^{\star}(s)|}\sum_{w^{\star} \in W^{\star}(s)}\widetilde{E}(w^{\star}w_1^{\star}F^{\star})R_{w^{\star}w_1^{\star}}^{\bG}(s).
\end{equation*}
Moreover, if $\mathfrak{C} \subseteq W^{\star}(s)$ is a two-sided cell then we define a set
\begin{equation*}
\mathcal{E}_0(\bG^F,s,\mathfrak{C}) = \{ \chi \in \Irr(\bG^F) \mid \langle \chi, \mathcal{R}^{\bG}_{\bT_0^{\star}}(\widetilde{E},s)\rangle_{\bG^F} \neq 0 \text{ and }E \in \Irr(W^{\star}(s) \mid \mathfrak{C})^{w_1^{\star}F^{\star}}\}.
\end{equation*}
As is explained in \cite[14.7]{taylor:2016:GGGRs-small-characteristics} this yields a disjoint union
\begin{equation*}
\mathcal{E}_0(\bG^F,s) = \bigsqcup_{\mathfrak{C} \subseteq W^{\star}(s)} \mathcal{E}_0(\bG^F,s,\mathfrak{C}),
\end{equation*}
where the union is taken over all the $w_1^{\star}F^{\star}$-stable two-sided cells.
\end{pa}

\subsection{Twisted Induction induces a Bijection}
\begin{definition}
If $A \subseteq \bG^{\star}$ is a subset containing a maximal torus of $
\bG^{\star}$ then we define the \emph{Levi cover} of $A$ to be the intersection of all Levi subgroups of $\bG^{\star}$ containing $A$.
\end{definition}

\begin{pa}
Note that the Levi cover is a Levi subgroup of $\bG^{\star}$ because $A$ contains a maximal torus of $\bG^{\star}$ and the intersection of two Levi subgroups containing a common maximal torus is a Levi subgroup, see \cite[2.1(i)]{digne-michel:1991:representations-of-finite-groups-of-lie-type}. Moreover, the Levi cover is clearly the unique minimal Levi subgroup containing $A$ with respect to inclusion. Hence, if $s \in \bG^{\star}$ is a semisimple element then this implies that $s$ is quasi-isolated in the Levi cover of $C_{\bG^{\star}}(s)$. Now, assume $s \in \bT_0^{\star}$ then the Levi cover contains $\bT_0^{\star}$ so it is conjugate to a standard Levi subgroup $\bL_I^{\star}$, for some subset $I \subseteq \Delta$, by an element of $N_{\bG^{\star}}(\bT_0^{\star})$. In particular, after possibly replacing $s$ by an $N_{\bG^{\star}}(\bT_0^{\star})$-conjugate we may assume that the Levi cover of $C_{\bG^{\star}}(s)$ is a standard Levi subgroup.
\end{pa}

\begin{pa}\label{pa:Levi-stable}
Let us continue our assumption that $s \in \bT_0^{\star}$ and the Levi cover of $C_{\bG^{\star}}(s)$ is a standard Levi subgroup. If $w^{\star} \in T_{W^{\star}}(s,F^{\star})$ then ${}^{n_{w^{\star}}}F^{\star}(s) = s$ so ${}^{n_{w^{\star}}}F^{\star}(C_{\bG^{\star}}(s)) = C_{\bG^{\star}}(s)$. A subgroup $\bL^{\star} \leqslant \bG^{\star}$ is a Levi subgroup containing $\bT_0^{\star}$ if and only if ${}^{n_{w^{\star}}}F^{\star}(\bL^{\star})$ is a Levi subgroup containing $\bT_0^{\star}$. Hence if $\bL_I^{\star}$ is the Levi cover of $C_{\bG^{\star}}(s)$ then we must have ${}^{n_{w^{\star}}}F^{\star}(\bL_I^{\star}) = \bL_I^{\star}$, or equivalently $w^{\star} \in T_{W^{\star}}(\bL_I^{\star},F^{\star}) = T_{W^{\star}}(W_I^{\star},F^{\star})$. Hence, we have an embedding $T_{W^{\star}}(s,F^{\star}) \hookrightarrow T_{W^{\star}}(W_I^{\star},F^{\star})$. Each element of $T_{W^{\star}}(W_I^{\star},F^{\star})$ may be factored uniquely as a product $x^{\star}w_1^{\star}$ with $x^{\star} \in W_I^{\star}$ and $w_1^{\star} \in T_{W^{\star}}(I^{\star},F^{\star})$, see the arguments used in the proof of \cite[4.2]{digne-michel:1991:representations-of-finite-groups-of-lie-type}. Hence we have a bijection $W_I^{\star} \backslash T_{W^{\star}}(W_I^{\star},F^{\star}) \to T_{W^{\star}}(I^{\star},F^{\star})$ and thus a well-defined map
\begin{equation}\label{eq:bijection-dual}
\mathcal{A}_{W^{\star}}(s,F^{\star}) \to W_I^{\star}\backslash T_{W^{\star}}(W_I^{\star},F^{\star}) \to T_{W^{\star}}(I^{\star},F^{\star}),
\end{equation}
where the first map is given by $W^{\star\circ}(s)x^{\star} \mapsto W_I^{\star}x^{\star}$.
\end{pa}

\begin{thm}[Lusztig]\label{thm:Lusztig-bijection}
Assume $s \in \bT_0^{\star}$ is a semisimple element such that $T_{W^{\star}}(s,F^{\star}) \neq\emptyset$ and $\bL_I^{\star}$ is the Levi cover of $C_{\bG^{\star}}(s)$ with $I \subseteq \Delta$. Let $w_1^{\star} \in T_{W^{\star}}(I^{\star},F^{\star})$ then the map $\mathcal{A}_{W_I^{\star}}(s,w_1^{\star}F^{\star}) \to \mathcal{A}_{W^{\star}}(s,F^{\star})$ given by $a^{\star} \mapsto a^{\star}w_1^{\star}$ is a bijection and the map $(-1)^{\ell(w_1)} R_{I,w_1}^{\bG} : \Irr(\bL_I^{Fw_1}) \to \mathbb{Z}\Irr(\bG^F)$ gives bijections
\begin{equation*}
\mathcal{E}_0(\bL_I^{Fw_1},s)\to\mathcal{E}_0(\bG^F,s) \qquad\text{and}\qquad \mathcal{E}_0(\bL_I^{Fw_1},s,a^{\star}) \to\mathcal{E}_0(\bG^F,s,a^{\star}w_1^{\star})
\end{equation*}
for any $a^{\star} \in \mathcal{A}_{W_I^{\star}}(s,w_1^{\star}F^{\star})$. Here $\ell : W \to \mathbb{Z}_{\geqslant 0}$ denotes the length function of $W$ determined by the reflections associated to the simple roots $\Delta$.
\end{thm}

\begin{proof}
Let $g \in \bG$ be such that $g^{-1}F(g) = F(n_{w_1})$ and set $\bP = {}^g\bP_I$ and $\bL = {}^g\bL_I$ then $\bL \leqslant \bP$ is an $F$-stable Levi complement of the parabolic subgroup $\bP$. The anti-isomorphism ${}^{\star} : W \to W^{\star}$ maps $Z_W(I,F)$ onto $T_{W^{\star}}(I^{\star},F^{\star})$, so the endomorphism $Fw_1$ stabilises $\bB_I$ and $\bT_0$. In particular, if $\mathcal{L} = (\bL,{}^g\bB_I,{}^g\bT_0)$ then $\imath_g \in \Iso((\mathcal{L}_I,Fw_1),(\mathcal{L},F))$ is an isomorphism. Dually, let us choose an element $g^{\star} \in \bG^{\star}$ such that $g^{\star-1}F^{\star}(g^{\star}) = n_{w_1^{\star}}$ and set $\bP^{\star} := {}^{g^{\star}}\bP_I^{\star}$ and $\bL^{\star} := {}^{g^{\star}}\bL_I^{\star}$ then $\bL^{\star} \leqslant \bP^{\star}$ is an $F^{\star}$-stable Levi complement of $\bP^{\star}$. Setting $\mathcal{L}^{\star} = (\bL^{\star},{}^{g^{\star}}\bB_I^{\star},{}^{g^{\star}}\bT_0^{\star})$ we have the map $\imath_{g^{\star}}^{-1} \in \Iso((\mathcal{L}^{\star},F^{\star}),(\mathcal{L}_I^{\star},w_1^{\star}F^{\star}))$ is an isomorphism dual to $\imath_g$.

Assume $a^{\star} = W^{\star\circ}(s)z^{\star} \in \mathcal{A}_{W_I^{\star}}(s,w_1^{\star}F^{\star})$ and let $t = {}^ls \in \bL_I^{\star w_1^{\star} F^{\star}}$ be a conjugate of $s$ with $l \in \bL_I^{\star}$ such that $l^{-1}{}^{n_{w_1^{\star}}}F^{\star}(l) = n_{z^{\star}}$. We then have a map
\begin{equation}\label{eq:map-decomp}
\mathcal{E}_0(\bL_I^{Fw_1},s,a^{\star}) = \mathcal{E}(\bL_I^{Fw_1},t) \xrightarrow{-\circ\imath_g^{-1}} \mathcal{E}(\bL^{F},\imath_{g^{\star}}(t)) \xrightarrow{R_{\bL \subseteq \bP}^{\bG}} \mathcal{E}(\bG^{F},\imath_{g^{\star}}(t)) = \mathcal{E}_0(\bG^{F},s,a^{\star}w_1^{\star}).
\end{equation}
The first and last identifications between Lusztig series is simply \cref{lem:DL-geo-series-compare} together with the computation
\begin{equation*}
(g^{\star}l)^{-1}F^{\star}(g^{\star}l) = l^{-1}g^{\star-1}F^{\star}(g)F^{\star}(l) = l^{-1}n_{w_1^{\star}}(n_{w_1^{\star}}^{-1}ln_{z^{\star}}n_{w_1^{\star}}) = n_{z^{\star}}n_{w_1^{\star}}.
\end{equation*}
The computation of the image of the first map is \cref{prop:Lusztig-series-iso-image}. It's clear that $\bL^{\star}$ contains the centraliser $C_{\bG^{\star}}(\imath_{g^{\star}}(t))$ so by \cite[13.25(ii)]{digne-michel:1991:representations-of-finite-groups-of-lie-type} we have $(-1)^{\ell(w_1)}R_{\bL \subseteq \bP}^{\bG}$ is a bijection between the geometric Lusztig series labelled by $\imath_{g^{\star}}(t)$. However, applying \cite[Th\'eor\`eme 11.10]{bonnafe:2006:sln} one easily concludes the same statement holds for the rational Lusztig series. By \cref{lem:twisted-ind-DL-ind} the map in \cref{eq:map-decomp} is nothing other than $R_{I,w_1}^{\bG}$ so the result follows.
\end{proof}

\section{Equivariance of Twisted Induction}\label{sec:equivariance-twisted-induction}
\begin{assumption}
We now assume that $\sigma \in \Iso(\mathcal{G},F)$ is a bijective isogeny and $\sigma^{\star} \in \Iso(\mathcal{G}^{\star},F^{\star})$ is dual to $\sigma$, c.f., \cref{pa:dual-isogenies}.
\end{assumption}

\begin{pa}
The following result gives the equivariance of twisted induction with respect to bijective isogenies. We note that the corresponding result for Deligne--Lusztig induction is well known, see for instance \cite[13.22]{digne-michel:1991:representations-of-finite-groups-of-lie-type} and \cite[Corollary 2.3]{navarro-tiep-turull:2008:brauer-characters-with-cyclotomic}. Our version of this statement has the added advantage that we can explicitly compute the corresponding automorphism on the Levi subgroup. We note that our proof is similar in strategy to that of \cite[Corollary 2.3]{navarro-tiep-turull:2008:brauer-characters-with-cyclotomic}.
\end{pa}

\begin{prop}\label{prop:equivariance}
Let $I \subseteq \Delta$ be a subset of simple roots and let $w \in Z_W(\bL_I,F)$. Assume $z \in W$ is such that $z\sigma(I) = I$ and $zF(\sigma(w)z^{-1}) = F(w)$ then there exists an element $n \in N_{\bG}(\bT_0)$, representing $z \in W$, such that $nF(\sigma(n_w)n^{-1}) = F(n_w)$. Moreover, if $n \in N_{\bG}(\bT_0)$ is such an element then $n\sigma \in \Iso(\mathcal{L}_I,Fw)$ restricts to an automorphism of $\bL_I^{Fw}$ and for any $\chi \in \Irr(\bL_I^{Fw})$ we have
\begin{equation*}
{}^{\sigma}R_{I,w}^{\bG}(\chi) = R_{I,w}^{\bG}({}^{n\sigma}\chi).
\end{equation*}
\end{prop}

\begin{proof}
By assumption we have $n_zF(\sigma(n_w)n_z^{-1}n_w^{-1}) \in \bT_0$ because $\sigma(\bT_0) = \bT_0$ and $F(\bT_0) = \bT_0$. Applying the Lang--Steinberg theorem inside $\bT_0$ to the Steinberg endomorphism $Fw$ we see that there exists an element $s \in \bT_0$ such that
\begin{equation*}
n_zF(\sigma(n_w)n_z^{-1}n_w^{-1}) = s^{-1}F(n_wsn_w^{-1}),
\end{equation*}
which implies $nF(\sigma(n_w)n^{-1}) = F(n_w)$ where $n = sn_z$. This condition ensures that $n\sigma\circ Fw = Fw\circ n\sigma$ so $n\sigma$ restricts to an automorphism of $\bL_I^{Fw}$.

We claim that the bijective morphism $\phi : \bG \to \bG$, defined by $\phi(x) = \sigma(x)n^{-1}$, restricts to a bijective morphism of varieties $\bY_{I,w}^{\bG} \to \bY_{I,w}^{\bG}$. Indeed, assume $x \in \bY_{I,w}^{\bG}$ then
\begin{equation*}
(\sigma(x)n^{-1})^{-1}F(\sigma(x)n^{-1}) = n\sigma(x^{-1}F(x))F(n^{-1}) \in n\sigma(F(n_w\bU_I))F(n^{-1}).
\end{equation*}
Clearly we have
\begin{equation*}
\sigma(F(n_w\bU_I)) = F(\sigma(n_w)\bU_{\sigma(I)}) = F(\sigma(n_w)\bU_{z^{-1}(I)}) = F(\sigma(n_w)n^{-1}\bU_In)
\end{equation*}
because $\sigma(\bP_I) = \bP_{\sigma(I)}$ and $n$ represents $z$. Therefore, we must have
\begin{equation*}
n\sigma(F(n_w\bU_I))F(n^{-1}) = nF(\sigma(n_w)n^{-1}\bU_I) = F(n_w\bU_I)
\end{equation*}
by the first part. This proves the claim.

A quick calculation shows that $\phi(gxl) = \sigma(g)\phi(x){}^n\sigma(l)$ for any $g \in \bG^F$, $x \in \bY_{I,w}^{\bG}$, and $l \in \bL_I^{Fw}$, from which we deduce that $\phi$ induces an isomorphism ${}^{\sigma}H_c^i(\bY_{I,w}^{\bG})^{n\sigma} \to H_c^i(\bY_{I,w}^{\bG})$ of $(\Ql\bG^F,\Ql\bL_I^{Fw})$-bimodules for any $i \in \mathbb{Z}$. We may now argue as in \cref{lem:twisted-ind-DL-ind} to complete the proof.
\end{proof}

\begin{pa}\label{pa:stab-Lusztig-series}
We wish to understand the effect of $\sigma$ on a Lusztig series $\mathcal{E}_0(\bG^F,s,a^{\star})$ with $s \in \bT_0^{\star}$ and $a^{\star} \in \mathcal{A}_{W^{\star}}(s,F^{\star})$. After possibly replacing $(s,a^{\star})$ by a pair in the same $W^{\star}$-orbit we may, and will, assume that the Levi cover of $C_{\bG^{\star}}(s)$ is a standard Levi subgroup $\bL_I^{\star}$ with $I \subseteq \Delta$. By \cref{prop:Lusztig-series-iso-image,lem:DL-geo-series-compare} we have ${}^{\sigma}\mathcal{E}_0(\bG^F,s,a^{\star}) = \mathcal{E}_0(\bG^F,\sigma^{\star-1}(s),\sigma^{\star-1}(a^{\star}))$; note here we've used that $F^{\star}$ and $\sigma^{\star}$ commute. As we will be interested in characters which are fixed by $\sigma$ we will assume that ${}^{\sigma}\mathcal{E}_0(\bG^F,s,a^{\star}) = \mathcal{E}_0(\bG^F,s,a^{\star})$. By \cref{pa:disjointness} there thus exists an element $x^{\star} \in W^{\star}$ such that $(s,a^{\star}) = ({}^{n_{x^{\star-1}}}\sigma^{\star-1}(s),x^{\star-1}\sigma^{\star-1}(a^{\star})F^{\star}(x^{\star}))$.
\end{pa}

\begin{pa}\label{pa:stab-Lusztig-series-1}
The exact same arguments that were used in \cref{pa:Levi-stable} show that ${}^{n_{x^{\star-1}}}\sigma^{\star-1}(\bL_I^{\star}) = \bL_I^{\star}$. Moreover, there exists an element $y^{\star} \in W_I^{\star}$ such that $y^{\star-1}x^{\star-1}\sigma^{\star-1}(I^{\star}) = I^{\star}$. Let us set $z^{\star} = x^{\star}y^{\star}$ then we have a bijection $\psi : T_{W^{\star}}(I^{\star},F^{\star}) \to T_{W^{\star}}(I^{\star},F^{\star})$ given by $\psi(w^{\star}) = z^{\star-1}\sigma^{\star-1}(w^{\star})F^{\star}(z^{\star})$. Moreover, we have a commutative diagram
\begin{equation*}
\begin{tikzcd}
\mathcal{A}_{W^{\star}}(s,F^{\star}) \arrow{d}{\phi}\arrow{r}{} & T_{W^{\star}}(I^{\star},F^{\star}) \arrow{d}{\psi}\\
\mathcal{A}_{W^{\star}}(s,F^{\star}) \arrow{r} & T_{W^{\star}}(I^{\star},F^{\star})
\end{tikzcd}
\end{equation*}
where the horizontal maps are given by \cref{eq:bijection-dual} and the map $\phi$ is as in \cref{pa:disjointness}. In particular, if $a^{\star} \in \mathcal{A}_{W^{\star}}(s,F^{\star})^{\phi}$ corresponds to $w_1^{\star} \in T_{W^{\star}}(I^{\star},F^{\star})$ under the map in \cref{eq:bijection-dual} then we have $z^{\star-1}\sigma^{\star-1}(w_1^{\star})F^{\star}(z^{\star}) = w_1^{\star}$. By duality this implies that $F^{-1}(z)\sigma(w_1)z^{-1} = w_1$ or, equivalently, that $zF(\sigma(w_1)z^{-1}) = F(w_1)$. From $z^{\star-1}\sigma^{\star-1}(I^{\star}) = I^{\star}$ we also get $z\sigma(I) = I$; thus we are in the setting of \cref{prop:equivariance} from which we obtain the following.
\end{pa}

\begin{thm}\label{thm:equivariance-DL-induction}
Assume $\mathcal{E}_0(\bG^F,s,a^{\star})$ is a $\sigma$-invariant Lusztig series; we will maintain the notation of \cref{pa:stab-Lusztig-series,pa:stab-Lusztig-series-1}. Let $n \in N_{\bG}(\bT_0)$ be an element, representing $z \in W$, such that $nF(\sigma(n_{w_1})n^{-1}) = F(n_{w_1})$ then we have bijections
\begin{equation*}
\mathcal{E}_0(\bL_I^{Fw_1},s)\to\mathcal{E}_0(\bG^F,s)
\qquad\text{and}\qquad
\mathcal{E}_0(\bL_I^{Fw_1},s,a^{\star}w_1^{\star-1})\to\mathcal{E}_0(\bG^F,s,a^{\star})
\end{equation*}
which are equivariant with respect to $\sigma$ and $n\sigma$ in the sense that
\begin{equation*}
{}^{\sigma}R_{I,w_1}^{\bG}(\chi) = R_{I,w_1}^{\bG}({}^{n\sigma}\chi)
\end{equation*}
for any $\chi \in \mathcal{E}_0(\bL_I^{Fw_1},s,a^{\star}w_1^{\star-1})$.
\end{thm}

\begin{proof}
The statement is simply \cref{thm:Lusztig-bijection,prop:equivariance}. We only need to check that the automorphism $n\sigma$ of $\bL_I^{Fw_1}$ preserves the Lusztig series $\mathcal{E}_0(\bL_I^{Fw_1},s,a^{\star}w_1^{\star-1})$. However, clearly $\sigma^{\star}z^{\star}$ is dual to the isogeny $n\sigma$ because $n$ represents $z$ so, as in \cref{pa:stab-Lusztig-series}, we have
\begin{equation*}
{}^{n\sigma}\mathcal{E}_0(\bL_I^{Fw_1},s,a^{\star}w_1^{\star-1}) = \mathcal{E}_0(\bL_I^{Fw_1},{}^{n_{z^{\star-1}}}\sigma^{\star-1}(s),z^{\star-1}\sigma^{\star-1}(a^{\star}w_1^{\star-1})).
\end{equation*}
Note first that ${}^{n_{z^{\star}}^{-1}}\sigma^{\star-1}(s) = {}^{n_{y^{\star}}^{-1}}s$, because ${}^{n_{x^{\star}}^{-1}}\sigma^{\star-1}(s) = s$, and as $\sigma^{\star-1}(w_1^{\star-1}) = F^{\star}(z^{\star})w_1^{\star-1}z^{\star-1}$ we have
\begin{align*}
z^{\star-1}\sigma^{\star-1}(a^{\star}w_1^{\star-1}) &= z^{\star-1}\sigma^{\star-1}(a^{\star})F^{\star}(z^{\star})w_1^{\star-1}\\
&= y^{\star-1}a^{\star}F^{\star}(y^{\star})w_1^{\star-1}\\
&= y^{\star-1}(a^{\star}w_1^{\star-1})w_1^{\star}F^{\star}(y^{\star}).
\end{align*}
As $y^{\star} \in W_I^{\star}$ this shows that the pair $({}^{n_{z^{\star}}^{-1}}\sigma^{\star-1}(s),z^{\star-1}\sigma^{\star-1}(a^{\star}w_1^{\star-1})) \in \mathcal{T}_{\bL_I^{\star}}(\bT_0^{\star},w_1^{\star}F^{\star})$ is in the same $W_I^{\star}$-orbit as $(s,a^{\star}w_1^{\star-1}) \in \mathcal{T}_{\bL_I^{\star}}(\bT_0^{\star},w_1^{\star}F^{\star})$. Thus, by \cref{pa:disjointness}, we have the Lusztig series is preserved.
\end{proof}

\section{Regular Embeddings}\label{sec:regular-embeddings}
\begin{pa}
We now assume that $\iota : \bG \to \widetilde{\bG}$ is a regular embedding, in the sense of \cite[7]{lusztig:1988:reductive-groups-with-a-disconnected-centre}. In particular, $\widetilde{\bG}$ is a connected reductive algebraic group with connected centre and $\iota$ is a closed embedding such that the derived subgroup of the image $\iota(\bG)$ is the derived subgroup of $\widetilde{\bG}$. Moreover, we assume $\widetilde{\bG}$ is equipped with a Steinberg endomorphism, also denoted by $F : \widetilde{\bG} \to \widetilde{\bG}$, such that $F\circ\iota = \iota\circ F$. This means that $\iota$ restricts to an embedding $\iota : \bG^F \to \widetilde{\bG}^F$ and the image $\iota(\bG^F)$ is normal in $\widetilde{\bG}^F$. In what follows we implicitly identify $\bG$ with its image $\iota(\bG)$ in $\widetilde{\bG}$. In this vein we denote by $\Res_{\bG^F}^{\widetilde{\bG}^F}(\widetilde{\chi})$ the character $\widetilde{\chi}\circ\iota$ for any character $\widetilde{\chi}$ of $\widetilde{\bG}^F$.
\end{pa}

\begin{pa}
Given such a regular embedding we can construct a canonical triple $\widetilde{\mathcal{G}} = (\widetilde{\bG},\widetilde{\bB}_0,\widetilde{\bT}_0)$ by setting $\widetilde{\bB}_0 = \bB_0Z(\widetilde{\bG})$ and $\widetilde{\bT}_0 = \bT_0Z(\widetilde{\bG})$. Clearly we have $F \in \Iso(\widetilde{\mathcal{G}})$. We assume fixed a triple $\widetilde{\mathcal{G}}^{\star} = (\widetilde{\bG}^{\star},\widetilde{\bB}_0^{\star},\widetilde{\bT}_0^{\star})$ dual to $\mathcal{G}$ and a Steinberg endomorphism $F^{\star} \in \Iso(\widetilde{\mathcal{G}}^{\star})$ dual to $F$. The embedding $\iota$ determines a surjective homomorphism $\iota^{\star} : \widetilde{\bG}^{\star} \to \bG^{\star}$, whose kernel is a central torus, satisfying $F^{\star}\circ\iota^{\star} = \iota^{\star}\circ F^{\star}$. We will denote by $\widetilde{W} = W_{\widetilde{\bG}}(\widetilde{\bT}_0)$ and $\widetilde{W}^{\star} = W_{\widetilde{\bG}^{\star}}(\widetilde{\bT}_0^{\star})$ the corresponding Weyl groups of $\widetilde{\bG}$ and $\widetilde{\bG}^{\star}$. Note that $\iota$, resp., $\iota^{\star}$, induces an isomorphism $W \to \widetilde{W}$, resp., $\widetilde{W}^{\star} \to W^{\star}$, and we will implicitly identify these groups through this isomorphism. With this we consider the labelling sets $\mathcal{T}_{\bG^{\star}}(\bT_0^{\star},F^{\star})$ and $\mathcal{T}_{\widetilde{\bG}^{\star}}(\widetilde{\bT}_0^{\star},F^{\star})$ defined in \cref{pa:disjointness}.
\end{pa}

\begin{lem}\label{lem:surjective-on-pairs}
The map $\iota^{\star}$ induces a surjective map $\iota^{\star} : \mathcal{T}_{\widetilde{\bG}^{\star}}(\widetilde{\bT}_0^{\star},F^{\star}) \to \mathcal{T}_{\bG^{\star}}(\bT_0^{\star},F^{\star})$.
\end{lem}

\begin{proof}
Let $\widetilde{s} \in \widetilde{\bT}_0^{\star}$ then we certainly have $\iota^{\star}$ induces an injective map $T_{\widetilde{W}^{\star}}(\widetilde{s},F^{\star}) \to T_{W^{\star}}(\iota^{\star}(\widetilde{s}),F^{\star})$. As the centre of $\widetilde{\bG}$ is connected we have $C_{\widetilde{\bG}^{\star}}(\widetilde{s}) = C_{\widetilde{\bG}^{\star}}^{\circ}(\widetilde{s})$, see \cite[13.15(ii)]{digne-michel:1991:representations-of-finite-groups-of-lie-type}, so $\widetilde{W}^{\star}(\widetilde{s}) = \widetilde{W}^{\star\circ}(\widetilde{s})$ and $\iota^{\star}(\widetilde{W}^{\star}(\widetilde{s})) = W^{\star\circ}(\iota^{\star}(\widetilde{s}))$. Hence $\iota^{\star}$ induces a map $\mathcal{A}_{\widetilde{W}^{\star}}(\widetilde{s},F^{\star}) \to \mathcal{A}_{W^{\star}}(\iota^{\star}(\widetilde{s}),F^{\star})$ so also a map $\mathcal{T}_{\widetilde{\bG}^{\star}}(\widetilde{\bT}_0^{\star},F^{\star}) \to \mathcal{T}_{\bG^{\star}}(\bT_0^{\star},F^{\star})$.

Assume $(s,a^{\star}) \in \mathcal{T}_{\bG^{\star}}(\bT_0^{\star},F^{\star})$ then as $\iota^{\star}$ is surjective there exists an element $\widetilde{s} \in \widetilde{\bT}_0^{\star}$ such that $\iota^{\star}(\widetilde{s}) = s$. Let $w^{\star} \in W^{\star}(s)$ be such that $a^{\star} = W^{\star\circ}(s)w^{\star}$ then ${}^{n_{w^{\star}}}F^{\star}(s) = s$ so $\widetilde{s}^{-1}{}^{n_{w^{\star}}}F^{\star}(\widetilde{s}) \in \Ker(\iota^{\star})$. As $\Ker(\iota^{\star})$ is connected and central we have by the Lang--Steinberg theorem that there exists an element $z \in \Ker(\iota^{\star})$ such that ${}^{w^{\star}}F^{\star}(\widetilde{s}z) = \widetilde{s}z$. This shows the map is surjective.
\end{proof}

\begin{rem}
For any element $\widetilde{s} \in \widetilde{\bT}_0^{\star}$ we have $\mathcal{A}_{\widetilde{W}^{\star}}(\widetilde{s},F^{\star})$ is a singleton. Hence we have a bijection between $\mathcal{T}_{\widetilde{\bG}^{\star}}(\widetilde{\bT}_0^{\star},F^{\star})$ and the $F^{\star}$-stable $\widetilde{W}^{\star}$-orbits on $\widetilde{\bT}_0^{\star}$. As mentioned in \cref{pa:geometric-series} we thus have $\mathcal{E}_0(\widetilde{\bG}^F,\widetilde{s},a^{\star}) = \mathcal{E}_0(\widetilde{\bG}^F,\widetilde{s})$ for any pair $(\widetilde{s},a^{\star}) \in \mathcal{T}_{\widetilde{\bG}^{\star}}(\widetilde{\bT}_0^{\star},F^{\star})$.
\end{rem}

\begin{lem}\label{lem:restriction}
Assume $(\widetilde{s},a^{\star}) \in \mathcal{T}_{\widetilde{\bG}^{\star}}(\widetilde{\bT}_0^{\star},F^{\star})$ and let $\iota^{\star}(\widetilde{s},a^{\star}) = (s,a^{\star})$.
\begin{enumerate}
	\item If $\widetilde{\chi} \in \mathcal{E}_0(\widetilde{\bG}^F,\widetilde{s},a^{\star})$ then each irreducible constituent of $\widetilde{\chi}$ is contained in the series $\mathcal{E}_0(\bG^F,s,a^{\star})$.
	\item If $\chi \in \mathcal{E}_0(\bG^F,s,a^{\star})$ then there exists a character $\widetilde{\chi} \in \mathcal{E}_0(\widetilde{\bG}^F,\widetilde{s},a^{\star})$ such that $\chi$ is a constituent of $\Res_{\bG^F}^{\widetilde{\bG}^F}(\widetilde{\chi})$.
\end{enumerate}
\end{lem}

\begin{proof}
Let $g \in \widetilde{\bG}$ be such that $\overline{g^{-1}F^{\star}(g)} \in A_{\widetilde{\bG}^{\star}}(\widetilde{s},F^{\star})$ corresponds to $a^{\star} \in \mathcal{A}_{\widetilde{W}^{\star}}(\widetilde{s},F^{\star})$ and set $\widetilde{t} = {}^g\widetilde{s}$ and $t = \iota^{\star}(\widetilde{t}) = {}^{\iota^{\star}(g)}s$. By \cref{lem:DL-geo-series-compare} we have $\mathcal{E}(\widetilde{\bG}^F,\widetilde{t}) = \mathcal{E}_0(\widetilde{\bG}^F,\widetilde{s})$ and $\mathcal{E}(\bG^F,t) = \mathcal{E}_0(\bG^F,s,a^{\star})$. The statement is now \cite[11.17]{bonnafe:2006:sln}.
\end{proof}

\section{Generalised Gelfand--Graev Representations (GGGRs)}\label{sec:GGGRs}
\begin{assumption}
From this point forward we assume that $p$ is a good prime for $\bG$ and $F$ is a Frobenius endomorphism endowing $\bG$ with an $\mathbb{F}_q$-rational structure. Moreover, we assume that $\bG$ is proximate in the sense of \cite[2.10]{taylor:2016:GGGRs-small-characteristics}. Recall this means a simply connected covering of the derived subgroup of $\bG$ is a separable morphism.
\end{assumption}

\begin{pa}
If $r = p^a$, with $a \geqslant 0$ an integer, then we denote by $F_r : \mathbb{K} \to \mathbb{K}$ the Frobenius endomorphism defined by $F_r(k) = k^r$; note that $F_1$ is the identity. Moreover, we define $\Frob_r(\mathcal{G}) \subseteq \Iso(\mathcal{G})$ to be the set of all isogenies $\sigma \in \Iso(\mathcal{G})$ such that $X(\sigma) = r\tau$ where $\tau : X(\bT_0) \to X(\bT_0)$ is a finite order automorphism preserving the set of simple roots $\Delta$. Moreover, we set $\Frob(\mathcal{G}) = \bigcup_{a \geqslant 0} \Frob_{p^a}(\mathcal{G})$. Any element of $\Frob_1(\mathcal{G})$ is an automorphism of $\bG$ and any element of $\Frob(\mathcal{G})\setminus \Frob_1(\mathcal{G})$ is a Frobenius endomorphism. Moreover, up to composition with an inner automorphism, $\Frob(\mathcal{G})$ contains every Frobenius endomorphism and automorphism of $\bG$. We write $\Frob(\mathcal{G},F) \subseteq \Frob(\mathcal{G})$ and $\Frob_r(\mathcal{G},F) \subseteq \Frob_r(\mathcal{G})$ for the subset of isogenies commuting with $F$.
\end{pa}

\begin{pa}
We will denote by $\lie{g}$ the Lie algebra of $\bG$ which we define to be the set of all $\mathbb{K}$-derivations $\Der_{\mathbb{K}}(\mathbb{K}[\bG],\mathbb{K})$, where $\mathbb{K}[\bG]$ is the affine algebra of $\bG$ and $\mathbb{K}$ is considered as a $\mathbb{K}[\bG]$-module via $f\cdot k = f(1)k$. We wish to now describe how an element $\sigma \in \Frob(\mathcal{G})$ induces a corresponding map of the Lie algebra. First, let us assume $\sigma \in \Frob_r(\mathcal{G})$ is a Frobenius endomorphism, so $r > 1$, then there exists an $\mathbb{F}_r$-subalgebra $\mathbb{K}[\bG]_{\sigma} = \{f \in \mathbb{K}[\bG] \mid \sigma^{\star}(f) = f^r\} \subseteq \mathbb{K}[\bG]$, where $\sigma^* : \mathbb{K}[\bG] \to \mathbb{K}[\bG]$ is the corresponding comorphism, such that the natural product map $\mathbb{K} \otimes_{\mathbb{F}_r} \mathbb{K}[\bG]_{\sigma} \to \mathbb{K}[\bG]$ is an isomorphism. If we set $\lie{g}_{\sigma}(\mathbb{F}_r) = \Der(\mathbb{K}[\bG]_{\sigma},\mathbb{F}_r)$ then we have an isomorphism $\lie{g} \cong \mathbb{K} \otimes_{\mathbb{F}_r} \lie{g}_{\sigma}(\mathbb{F}_r)$. With this isomorphism we can define a Frobenius endomorphism on $\lie{g}$, which we also denote by $\sigma : \lie{g} \to \lie{g}$. This endomorphism is of the form $\sigma = F_r \otimes \psi$, where $\psi : \lie{g}_{\sigma}(\mathbb{F}_r) \to \lie{g}_{\sigma}(\mathbb{F}_r)$ is such that the corresponding comorphism $\psi^* : \mathbb{K}[\lie{g}_{\sigma}(\mathbb{F}_r)] \to \mathbb{K}[\lie{g}_{\sigma}(\mathbb{F}_r)]$ is given by $\psi^*(f) = f^r$. Note via this process we have $F : \bG \to \bG$ induces a Frobenius endomorphism $F : \lie{g} \to \lie{g}$.
\end{pa}

\begin{pa}
Now if $\sigma \in \Frob_1(\mathcal{G})$ is an automorphism then the differential $\mathrm{d}\sigma : \lie{g} \to \lie{g}$ is an automorphism of the Lie algebra; we take this differential to be our induced map which we will also sloppily denote by $\sigma$. If $\sigma$ is a Frobenius endomorphism then the differential $\mathrm{d}\sigma$ is $0$; thus these two definitions of the induced map are different. We will denote by $\widecheck{X}(\bG)$ the set of all cocharacters $\lambda : \mathbb{K}^{\times} \to \bG$. To each such cocharacter $\lambda \in \widecheck{X}(\bG)$ we have a corresponding parabolic subgroup $\bP(\lambda)$ with unipotent radical $\bU(\lambda)$ and Levi complement $\bL(\lambda) = C_{\bG}(\lambda(\mathbb{K}^{\times}))$, see \cite[8.4.5]{springer:2009:linear-algebraic-groups}.
\end{pa}

\begin{pa}
Let us denote by $\mathcal{U} \subseteq \bG$ the unipotent variety and by $\mathcal{N} \subseteq \lie{g}$ the nilpotent cone. A Springer isomorphism $\phi_{\mathrm{spr}} : \mathcal{U} \to \mathcal{N}$ is a $\bG$-equivariant isomorphism of varieties, where the $\bG$-action on $\mathcal{U}$ is by conjugation and the $\bG$-action on $\mathcal{N}$ is via the adjoint representation; this commutes with the Frobenius $F$. In \cite[4.6]{taylor:2016:GGGRs-small-characteristics} it is shown, assuming that $p$ is a good prime and $\bG$ is proximate, that there exists a Springer isomorphism $\phi_{\mathrm{spr}} : \mathcal{U} \to \mathcal{N}$ whose restriction to each $\bU(\lambda)$ is a Kawanaka isomorphism, in the sense of \cite[4.1]{taylor:2016:GGGRs-small-characteristics}. We will need the following slight modification of the statements in \cite{taylor:2016:GGGRs-small-characteristics}; we omit the proof.
\end{pa}

\begin{prop}
We may assume $\phi_{\mathrm{spr}}$ satisfies $\phi_{\mathrm{spr}}\circ\sigma = \sigma\circ\phi_{\mathrm{spr}}$ for any $\sigma \in \Frob(\mathcal{G})$.
\end{prop}

\begin{pa}
Now assume $\sigma \in \Frob_r(\mathcal{G})$ then for any cocharacter $\lambda \in \widecheck{X}(\bG)$ we define a new cocharacter $\sigma\cdot\lambda \in \widecheck{X}(\bG)$ by setting
\begin{equation*}
(\sigma\cdot\lambda)(k) = \sigma(\lambda(F_r^{-1}(k)))
\end{equation*}
for all $k \in \mathbb{K}^{\times}$. We denote by $\widecheck{X}(\bG)^{\sigma}$ those cocharacters satisfying $\sigma\cdot\lambda = \lambda$. Recall that each cocharacter $\lambda \in \widecheck{X}(\bG)$ determines a grading $\lie{g} = \oplus_{i \in \mathbb{Z}} \lie{g}(\lambda,i)$ of the Lie algebra. For each integer $i > 0$ we have a Lie subalgebra $\lie{u}(\lambda,i) = \oplus_{j \geqslant i} \lie{g}(\lambda,j)$ and a corresponding closed connected subgroup $\bU(\lambda,i) \leqslant \bU(\lambda)$ whose Lie algebra is $\lie{u}(\lambda,i)$. The Levi subgroup $\bL(\lambda)$ acts on $\lie{g}(\lambda,2)$ and there is a unique open dense orbit which we denote by $\lie{g}(\lambda,2)_{\reg}$. Moreover, if $\sigma \in \Frob(\mathcal{G})$ then we have $\sigma(\bU(\lambda,i)) = \bU(\sigma\cdot\lambda,i)$ and $\sigma(\lie{u}(\lambda,i)) = \lie{u}(\sigma\cdot\lambda,i)$. In particular, if $\lambda \in \widecheck{X}(\bG)^{\sigma}$ then $\sigma(\bU(\lambda,i)) = \bU(\lambda,i)$ and $\sigma(\lie{u}(\lambda,i)) = \lie{u}(\lambda,i)$.
\end{pa}

\begin{pa}
Now assume $u \in \mathcal{U}^F$ is a rational unipotent element and $e = \phi_{\mathrm{spr}}(u) \in \mathcal{N}^F$ is the corresponding rational nilpotent element. In \cite[\S3]{taylor:2016:GGGRs-small-characteristics} we have defined a set of cocharacters $\mathcal{D}(\bG) \subseteq \widecheck{X}(\bG)$ which is invariant under the map $\lambda \mapsto \sigma\cdot\lambda$ for any $\sigma \in \Frob(\mathcal{G})$. Moreover, there exists an $F$-fixed cocharacter $\lambda \in \mathcal{D}(\bG)^F$, which is unique up to $\bG^F$-conjugacy, such that $e \in \lie{g}(\lambda,2)_{\reg}$. As above we have $F(\bU(\lambda,2)) = \bU(\lambda,2)$ and we wish to define a linear character $\varphi_u : \bU(\lambda,2)^F \to \Ql$. With this in mind let $\kappa : \lie{g} \times \lie{g} \to \mathbb{K}$ be a $\bG$-invariant trace form which is not too degenerate in the sense of \cite[5.6]{taylor:2016:GGGRs-small-characteristics}. We may assume that $\kappa$ satisfies $\kappa(\sigma(X),\sigma(Y)) = F_r(\kappa(X,Y))$ for any $\sigma \in \Frob_r(\mathcal{G})$, see \cite[I, 5.3]{springer-steinberg:1970:conjugacy-classes} and \cite[5.6]{taylor:2016:GGGRs-small-characteristics}. Note that if $\sigma \in \Frob_1(\mathcal{G})$ then the invariance of $\kappa$ under $\sigma$ follows from the fact that it is a trace form. Let us furthermore assume ${}^{\dag} : \lie{g} \to \lie{g}$ is the map defined in \cite[5.2]{taylor:2016:GGGRs-small-characteristics}. This is then an $\mathbb{F}_r$-opposition automorphism, for any prime power $r = p^a > 1$, in the sense of \cite[5.1]{taylor:2016:GGGRs-small-characteristics}. Moreover, we have $\sigma(X^{\dag}) = \sigma(X)^{\dag}$ for all $X \in \lie{g}$ and $\sigma \in \Frob(\mathcal{G})$.
\end{pa}

\begin{pa}
We now assume chosen once and for all a fixed linear character $\chi_p :  \mathbb{F}_p^+ \to \Ql^{\times}$ of the finite field $\mathbb{F}_p$ viewed as an additive group. For any power $r$ of $p$ we then obtain a linear character $\chi_r : \mathbb{F}_r^+ \to \Ql^{\times}$ defined by $\chi_r = \chi_p \circ \Tr_{\mathbb{F}_r/\mathbb{F}_p}$ where $\Tr_{\mathbb{F}_r/\mathbb{F}_p} : \mathbb{F}_r \to \mathbb{F}_p$ is the field trace. We note that as $\Tr_{\mathbb{F}_r/\mathbb{F}_p}\circ F_p = \Tr_{\mathbb{F}_r/\mathbb{F}_p}$ we have $\chi_r \circ F_p = \chi_r$. With this we define the linear character $\varphi_u : \bU(\lambda,2)^F \to \Ql$ by setting
\begin{equation*}
\varphi_u(x) = \chi_q(\kappa(e^{\dag},\phi_{\mathrm{spr}}(x)))
\end{equation*}
for any $x \in \bU(\lambda,2)^F$. This is a linear character because the restriction of $\phi_{\mathrm{spr}}$ to $\bU(\lambda,2)$ is a Kawanaka isomorphism, see \cite[\S4]{taylor:2016:GGGRs-small-characteristics}.
\end{pa}

\begin{definition}
If $\Ind_{\bU(\lambda,2)^F}^{\bG^F}(\varphi_u)$ denotes the induction of the linear character then
\begin{equation*}
\Gamma_u^{\bG^F} = q^{-\dim \lie{g}(\lambda,1)/2}\Ind_{\bU(\lambda,2)^F}^{\bG^F}(\varphi_u)
\end{equation*}
is the character of a representation of $\bG^F$ known as a \emph{generalised Gelfand--Graev representation (GGGR)}, see \cite[\S5]{taylor:2016:GGGRs-small-characteristics}.
\end{definition}

\begin{prop}\label{prop:invariant-GGGRs}
For any rational unipotent element $u \in \mathcal{U}^F$ and bijective isogeny $\sigma \in \Frob(\mathcal{G},F)$ we have ${}^{\sigma}\Gamma_u^{\bG^F} = \Gamma_{\sigma(u)}^{\bG^F}$. Furthermore, if $u \in \mathcal{U}^F$ and $\sigma(u)$ are $\bG^F$-conjugate then ${}^{\sigma}\Gamma_u^{\bG^F} = \Gamma_u^{\bG^F}$.
\end{prop}

\begin{proof}
We have $\sigma(e) \in \lie{g}(\sigma\cdot\lambda,2)_{\reg}$ and $\sigma\cdot\lambda \in \mathcal{D}(\bG)^F$ is $F$-fixed, because $\sigma$ commutes with $F$, so certainly
\begin{equation*}
{}^{\sigma}\Ind_{\bU(\lambda,2)^F}^{\bG^F}(\varphi_u) = \Ind_{\sigma(\bU(\lambda,2)^F)}^{\bG^F}({}^{\sigma}\varphi_u) = \Ind_{\bU(\sigma\cdot\lambda,2)^F}^{\bG^F}({}^{\sigma}\varphi_u).
\end{equation*}
For any $x \in \bU(\sigma\cdot\lambda,2)^F$ we have
\begin{equation*}
{}^{\sigma}\varphi_u(x) = \chi_q(\kappa(e^{\dag},\phi_{\mathrm{spr}}(\sigma^{-1}(x)))) = \chi_q(\kappa(\sigma(e)^{\dag},\phi_{\mathrm{spr}}(x))) = \varphi_{\sigma(u)}(x),
\end{equation*}
where the last equalities follow from the properties of $\kappa$, $\phi_{\mathrm{spr}}$, $\dag$, and $\chi_q$ listed above. As $\sigma \in \Frob_r(\mathcal{G},F)$ is bijective and $F_r$-semilinear on $\lie{g}$ we see that $\lie{g}(\lambda,1)$ and $\lie{g}(\sigma\cdot\lambda,1)$ have the same dimension. From this we conclude that ${}^{\sigma}\Gamma_u^{\bG^F} = \Gamma_{\sigma(u)}^{\bG^F}$. The second statement follows from the fact that $\Gamma_u^{\bG^F} = \Gamma_{gug^{-1}}^{\bG^F}$ for any $g \in \bG^F$, which follows easily from the construction of $\Gamma_u^{\bG^F}$.
\end{proof}

\section{Symplectic Groups}\label{sec:auts}
\begin{assumption}
From this point forward we assume that $\bG$ is the symplectic group $\Sp_{2n}(\mathbb{K})$ defined over $\mathbb{K} = \overline{\mathbb{F}}_p$ with $p$ an odd prime. The underlying alternating bilinear form defining $\Sp_{2n}(\mathbb{K})$ is chosen to be that of \cite[1.3.15]{geck:2003:intro-to-algebraic-geometry}.
\end{assumption}

\begin{pa}\label{pa:symp-setup}
We set $\mathcal{G} = (\bG,\bB_0,\bT_0)$, where $\bT_0 \leqslant \bB_0 \leqslant \bG$ are the maximal torus of diagonal matrices and Borel subgroup of upper triangular matrices respectively. A dual triple $\mathcal{G}^{\star} = (\bG^{\star},\bB_0^{\star},\bT_0^{\star})$ is obtained by taking $\bG^{\star}$ to be the special orthogonal group $\SO_{2n+1}(\mathbb{K})$, defined with respect to the underlying symmetric bilinear form given in \cite[1.3.15]{geck:2003:intro-to-algebraic-geometry}. Moreover $\bT_0^{\star} \leqslant \bB_0^{\star}$ are again defined to be the maximal torus of diagonal matrices and Borel subgroup of upper triangular matrices respectively.
\end{pa}

\begin{pa}
We assume $F_p \in \Iso(\mathcal{G})$ and $F_p^{\star} \in \Iso(\mathcal{G}^{\star})$ are the Frobenius endomorphisms raising each matrix entry to the power $p$. Note that $X(F_p) = p$ on $X(\bT_0)$ and $X(F_p^{\star}) = p$ on $X(\bT_0^{\star})$ so these isogenies are dual. We now choose a regular embedding $\iota : \bG \to \widetilde{\bG}$, as in \cref{sec:regular-embeddings}, with respect to $F_p$. We will use all the notation for $\widetilde{\mathcal{G}}$ and $\widetilde{\mathcal{G}}^{\star}$ introduced in \cref{sec:regular-embeddings}.
\end{pa}

\begin{pa}\label{pa:weyl-group-reps}
To apply the results from the previous sections, in particular \cref{thm:equivariance-DL-induction}, we need to choose for each element $w \in W = W_{\bG}(\bT_0)$ a representative $n_w \in N_{\bG}(\bT_0)$. For each simple reflection $s_{\alpha}$, with $\alpha \in \Delta$, we choose the representative given in \cite[1.7.3]{geck:2003:intro-to-algebraic-geometry}. For a general element $w \in W$ we assume chosen a reduced expression $s_1\cdots s_r$, with each $s_i$ a simple reflection, and then set $n_w = n_{s_1}\cdots n_{s_r}$. Choosing the representatives in this way we have $F_p(n_w) = n_w$ for all $w \in W$.
\end{pa}

\begin{pa}\label{pa:frob-autos}
The Frobenius endomorphism $F_p$, resp., $F_p^{\star}$, generates a cyclic subgroup of $\Aut(\widetilde{\bG})$ and $\Aut(\bG)$, resp., $\Aut(\widetilde{\bG}^{\star})$ and $\Aut(\bG^{\star})$, all of whose non-identity elements are Frobenius endomorphisms. Through $\iota$, resp., $\iota^{\star}$, we have the subgroup of $\Aut(\widetilde{\bG})$, resp., $\Aut(\widetilde{\bG}^{\star})$, generated by $F_p$ is mapped isomorphically onto the corresponding subgroup of $\Aut(\bG)$, resp., $\Aut(\bG^{\star})$. For each prime power $r = p^a > 1$ we set $F_r = F_p^a$ and $F_r^{\star} = {F_p^{\star}}^a$; these are dual Frobenius endomorphisms. Moreover we fix a prime power $q$ and set $F = F_q$ and $F^{\star} = F_q^{\star}$.
\end{pa}

\begin{pa}\label{pa:autos-symp}
Now it's clear that we have $F_p \in \Aut(\widetilde{\bG},F)$ so the cyclic subgroup generated by $F_p$ determines a subgroup $D \leqslant \Aut(\widetilde{\bG}^F)$, which is the subgroup of \emph{field automorphisms}. As $F_p$ preserves $\bG^F$ we thus have $D$ is mapped isomorphically onto a subgroup of $\Aut(\bG^F)$ via $\iota$; we also denote this subgroup by $D$. The group $\bG^F$ is normal in $\widetilde{\bG}^F$ and we have an injective homomorphism $\widetilde{\bG}^F/\bG^FZ(\widetilde{\bG}^F) \to \Aut(\bG^F)$ whose image consists of the diagonal automorphisms of $\bG^F$. We note that, ccording to \cite[2.5.1]{gorenstein-lyons-solomon:1998:classification-3}, we have $\Aut(\bG^F) \cong (\widetilde{\bG}^F/Z(\widetilde{\bG}^F)) \rtimes D$.
\end{pa}

\section{The Springer Correspondence}\label{sec:springer}
\begin{assumption}
From this point forward we assume $\sigma = F_r \in \Iso(\mathcal{G},F)$ is as in \cref{pa:frob-autos}.
\end{assumption}

\begin{pa}\label{pa:uni-classes}
To understand the effect of the automorphism $\sigma$ on Kawanaka's generalised Gelfand--Graev representations we need to understand the conjugacy classes of $\bG$. With this in mind let us denote by $\mathcal{P}_1(2n)$ the set of all partitions $\lambda = (1^{r_1},2^{r_2},\dots) \vdash 2n$ such that $r_{2i+1} \equiv 0 \pmod{2}$ for each $1 \leqslant i \leqslant n$. To each unipotent element $u \in \mathcal{U}$ one associates a partition $\lambda \in \mathcal{P}_1(2n)$ which is given by the sizes of the Jordan blocks in the Jordan normal form of $u$. The map $\mathcal{U} \to \mathcal{P}_1(2n)$ sending $u \mapsto \lambda$ then induces a bijection $\mathcal{U}/\bG \to \mathcal{P}_1(2n)$, see \cite[Corollary 3.6]{liebeck-seitz:2012:unipotent-nilpotent-classes}. We will need the following two lemmas concerning unipotent elements of $\bG$.
\end{pa}

\begin{lem}\label{lem:conj-unipotent-elms}
If $u \in \mathcal{U}^F$ is a rational unipotent element then $\sigma(u)$ and $u$ are $\bG^F$-conjugate.
\end{lem}

\begin{proof}
By the parameterisation in terms of the Jordan normal form it is clear that every unipotent conjugacy class $\mathcal{O} \in \mathcal{U}/\bG$ is $F_p$-stable. In particular, let $\mathcal{O}$ be the class containing $u$ and fix an element $u_0 \in \mathcal{O}^{F_p}$ then $u_0$ is fixed by $\sigma$ and $F$. Moreover, by \cite[7.1]{liebeck-seitz:2012:unipotent-nilpotent-classes} we may assume that $F_p$ acts trivially on the component group $A_{\bG}(u_0) = C_{\bG}(u_0)/C_{\bG}^{\circ}(u_0)$. Applying \cref{lem:perm-rat-orbits} we conclude that $\sigma(u)$ and $u$ are $\bG^F$-conjugate.
\end{proof}

\begin{lem}\label{lem:comp-grp}
Let $u \in \mathcal{U}$ be a unipotent element whose $\bG$-conjugacy class is parameterised by $\lambda = (1^{r_1},2^{r_2},\dots) \in \mathcal{P}_1(2n)$ then we have $A_{\widetilde{\bG}}(u) = C_{\widetilde{\bG}}(u)/C_{\widetilde{\bG}}^{\circ}(u) \cong (\mathbb{Z}/2\mathbb{Z})^{n(u)-\delta(u)}$ is an elementary abelian $2$-group where $n(u) = |\{i \in \mathbb{N} \mid r_{2i} \neq 0\}|$ and $\delta(u) = 1$ if $r_{2i} \equiv 1 \pmod{2}$ for some $i \in \mathbb{N}$ and $\delta(u) = 0$ otherwise.
\end{lem}

\begin{proof}
This follows from \cite[Theorem 3.1]{liebeck-seitz:2012:unipotent-nilpotent-classes} together with the fact that the natural map $A_{\bG}(u) \to A_{\widetilde{\bG}}(u)$ is surjective with kernel given by the image of the centre $Z(\bG)$.
\end{proof}

\begin{pa}
Let $r,s,n \in \mathbb{N}_0$ and $d \in \mathbb{Z}$ be integers and set $e = \lfloor d/2\rfloor = \max\{a \in \mathbb{Z} \mid a \leqslant d/2\}$. Following \cite{lusztig-spaltenstein:1985:on-the-generalised-springer-correspondence} we denote by $\widetilde{\mathbb{X}}_{n,d}^{r,s}$ the set of all ordered pairs $\sbpair{A}{B}$ of finite sequences $A = (a_1,\dots,a_{m+d})$ and $B = (b_1,\dots,b_m)$ of non-negative integers such that the following hold:
\begin{enumerate}
	\item $a_i - a_{i-1} \geqslant r+s$ for all $1 < i \leqslant m+d$,
	\item $b_i - b_{i-1} \geqslant r+s$ for all $1 < i \leqslant m$,
	\item $b_1 \geqslant s$,
	\item $\sum_{i=1}^{m+d} a_i + \sum_{j=1}^m b_j = n + r(m+e)(m+d-e-1) + s(m+e)(m+d-e)$.
\end{enumerate}
There is a natural shift operation $\widetilde{\mathbb{X}}_{n,d}^{r,s} \to \widetilde{\mathbb{X}}_{n,d}^{r,s}$ on this set defined by
\begin{equation*}
\bpair{A}{B} \mapsto \bpair{0,a_1+r+s,\dots,a_{m+d}+r+s}
{s,b_1+r+s,\dots,b_m+r+s}
\end{equation*}
This induces an equivalence relation on $\widetilde{\mathbb{X}}_{n,d}^{r,s}$ and we denote by $\mathbb{X}_{n,d}^{r,s}$ the resulting set of equivalence classes. We will write the equivalence class containing $\sbpair{A}{B}$ as $\ssymb{A}{B}$. We call $n$ the rank of $\ssymb{A}{B}$.
\end{pa}

\begin{rem}
It is easily checked that the set $\mathbb{X}_{n,d}^{0,0}$ is naturally in bijection with the set of all bipartitions of $n$, i.e., pairs of partitions whose entries sum to $n$.
\end{rem}

\begin{pa}
We now recall that we have an addition of symbols $\oplus : \mathbb{X}_{n,d}^{r,s} \times \mathbb{X}_{n',d}^{r',s'} \to \mathbb{X}_{n+n',d}^{r+r',s+s'}$ defined as follows. If $\ssymb{A}{B} \in \mathbb{X}_{n,d}^{r,s}$ and $\ssymb{A'}{B'} \in \mathbb{X}_{n',d}^{r',s'}$ then after shifting we may assume that $|B| = |B'| = m$ and $|A| = |A'| = m+d$. With this we define
\begin{equation*}
\symb{A}{B} \oplus \symb{A'}{B'} = \symb{a_1+a_1',\dots,a_{m+d}+a_{m+d}'}{b_1+b_1',\dots,b_m+b_m'}.
\end{equation*}
If $d \in \{0,1\}$ then the set $\mathbb{X}_{0,d}^{r,s}$ consists of a single element $\Lambda_{0,d}^{r,s}$; specifically $\Lambda_{0,0}^{r,s} = \ssymb{\emptyset}{\emptyset}$ and $\Lambda_{0,1}^{r,s} = \ssymb{0}{\emptyset}$. Adding with these symbols clearly leaves the rank unchanged.
\end{pa}

\begin{pa}\label{pa:chars-weyl}
Assume $k \geqslant 1$ is an integer and let $W_k' \leqslant W_k$ denote the Weyl groups of type $\D_k$ and $\B_k$ respectively. It is well known that we have a bijection $\mathbb{X}_{k,1}^{0,0} \to \Irr(W_k)$ between the bipartitions of $k$ and the irreducible characters of $W_k$; we assume this bijection is the one defined in \cite[4.5]{lusztig:1984:characters-of-reductive-groups}. Now let us denote by $\mathbb{Y}_{k,1}^{0,0} \subseteq \mathbb{X}_{k,1}^{0,0}$ the subset of symbols $\ssymb{A}{B}$ such that
\begin{equation*}
\sum_{i=1}^m a_i > \sum_{i=1}^m b_i.
\end{equation*}
where $A = (a_1,\dots,a_m)$ and $B = (b_1,\dots,b_m)$. We then have an injective map $\mathbb{Y}_{k,1}^{0,0} \to \Irr(W_k')$ defined by $\ssymb{A}{B} \to \Res_{W_k'}^{W_k}(\ssymb{A}{B})$. Note this map makes sense because $\ssymb{A}{B} \in \mathbb{X}_{k,1}^{0,0}$, which we identify with $\Irr(W_k)$, and the restrictions are irreducible in this case.
\end{pa}

\begin{rem}
Our choice of ordering for the symbols in $\mathbb{Y}_{k,1}^{0,0}$ ensures that the $b$-invariant of the character of $\Irr(W_k')$ labelled by $\ssymb{A}{B} \in \mathbb{Y}_{k,1}^{0,0}$ is the same as that of the character of $\Irr(W_k)$ labelled by $\ssymb{A}{B} \in \mathbb{X}_{k,1}^{0,0}$, see \cite[5.6.2]{geck-pfeiffer:2000:characters-of-finite-coxeter-groups}.
\end{rem}

\begin{pa}
Let us denote by $\mathcal{N}_{\bG}$ the set of all pairs $(\mathcal{O},\mathscr{E})$ consisting of a unipotent conjugacy class $\mathcal{O}$ of $\bG$ and an irreducible $\bG$-equivariant local system $\mathscr{E}$ on $\mathcal{O}$, taken up to isomorphism. In \cite{lusztig:1984:intersection-cohomology-complexes}, building on work of Springer, Lusztig has defined an injective map $\Spr : \Irr(W) \to \mathcal{N}_{\bG}$ called the Springer correspondence. We may combinatorially describe this map as follows. In \cite[\S12]{lusztig:1984:intersection-cohomology-complexes} Lusztig has defined a bijection
\begin{equation}\label{eq:unipotent-pairs}
\mathcal{N}_{\bG} \longleftrightarrow \mathbb{X}_n^{1,1} := \bigcup_{\substack{d \in \mathbb{Z}\\ d\text{ odd}}} \mathbb{X}_{n,d}^{1,1}.
\end{equation}
Identifying $\mathcal{N}_{\bG}$ with $\mathbb{X}_n^{1,1}$ we then have the image of $\Spr$ is the set $\mathbb{X}_{n,1}^{1,1}$. Identifying $\Irr(W)$ with $\mathbb{X}_{n,1}^{0,0}$, as in \cref{pa:chars-weyl}, we have the map $\Spr$ is given by
\begin{equation}\label{eq:spr-map}
\Spr(X) = X \oplus \Lambda_{0,1}^{1,1}
\end{equation}
for any $X \in \mathbb{X}_{n,1}^{0,0}$.
\end{pa}

\begin{rem}
Let $\mathcal{N}_{\widetilde{\bG}}$ be defined for $\widetilde{\bG}$ as $\mathcal{N}_{\bG}$ is defined for $\bG$. The embedding $\iota$ induces a bijection between the unipotent conjugacy classes of $\bG$ and those of $\widetilde{\bG}$. Moreover, let $\mathcal{O} \subseteq \bG \subseteq \widetilde{\bG}$ be a unipotent conjugacy class with class representative $u \in \mathcal{O}$. The natural embedding $C_{\bG}(u) \hookrightarrow C_{\widetilde{\bG}}(u)$ induces a surjective homomorphism $A_{\bG}(u) \to A_{\widetilde{\bG}}(u)$ between the corresponding component groups and thus an injective map $\Irr(A_{\widetilde{\bG}}(u)) \to \Irr(A_{\bG}(u))$. This corresponds to an injective map between the isomorphism classes of $\widetilde{\bG}$-equivariant irreducible local systems on $\mathcal{O}$ and the $\bG$-equivariant irreducible local systems on $\mathcal{O}$. In particular, we have $\iota$ defines an embedding $\mathcal{N}_{\widetilde{\bG}} \hookrightarrow \mathcal{N}_{\bG}$ and the image of $\Spr$ is contained in $\mathcal{N}_{\widetilde{\bG}}$. The map $\Spr$ also defines the Springer correspondence for $\widetilde{\bG}$.
\end{rem}

\section{Cuspidal Characters in Quasi-Isolated Series}\label{sec:mult-1}
\begin{assumption}
We now assume that $\widetilde{s} \in \widetilde{\bT}_0^{\star}$ is such that $\iota^{\star}(\widetilde{s})$ is quasi-isolated in $\bG^{\star}$ and $T_{\widetilde{W}^{\star}}(\widetilde{s},F^{\star}) \neq \emptyset$.
\end{assumption}

\begin{pa}
Assume $\widetilde{\chi} \in \Irr(\widetilde{\bG}^F)$ is an irreducible character. We recall that Lusztig has defined a corresponding integer $n_{\widetilde{\chi}} \geqslant 1$ such that $n_{\widetilde{\chi}}\cdot\widetilde{\chi}(1) \in \mathbb{Z}[q]$ is an integral polynomial in $q$, c.f., \cite[4.26]{lusztig:1984:characters-of-reductive-groups}. This is usually referred to as the generic denominator of $\widetilde{\chi}$. Moreover, we have a corresponding $F$-stable unipotent conjugacy class $\mathcal{O}_{\widetilde{\chi}}^* \subseteq \widetilde{\bG}$ called the \emph{wave front set} of $\widetilde{\chi}$, see \cite{kawanaka:1985:GGGRs-and-ennola-duality,lusztig:1992:a-unipotent-support,taylor:2016:GGGRs-small-characteristics}. If $u \in \mathcal{O}_{\widetilde{\chi}}^*$ then it is known that $n_{\widetilde{\chi}}$ divides $|A_{\widetilde{\bG}}(u)|$. We aim to show that, for certain cuspidal characters, we have an equality $n_{\widetilde{\chi}} = |A_{\widetilde{\bG}}(u)|$. This numerical result will be crucial for showing that these cuspidal characters are invariant under $\sigma$. We note that a statement similar to this has been used by Geck in relation to cuspidal character sheaves, see \cite[5.3]{geck:1999:character-sheaves-and-GGGRs}.
\end{pa}

\begin{pa}
To obtain this result we need to describe the possible sets $\mathcal{E}_0(\widetilde{\bG}^F,\widetilde{s},\widetilde{\mathfrak{C}})$ which contain a cuspidal character. By the classification of quasi-isolated semisimple elements, see \cite[4.11]{bonnafe:2005:quasi-isolated}, we have $s \in \bT_0^{\star}$ is quasi-isolated in $\bG^{\star}$ if and only if $s^2 = 1$. In particular, if $V = \mathbb{K}^{2n+1}$ is the natural module for $\bG^{\star}$ then we have $C_{\bG^{\star}}^{\circ}(s)$ is of type $\B_a\D_b$ where $a = \dim \Ker(s-\ID_V)$ and $b = \dim \Ker(s+\ID_V)$. Note that a factor $\D_1$ should be considered empty. The centraliser $C_{\widetilde{\bG}^{\star}}(\widetilde{s})$ is mapped isomorphically onto the connected component $C_{\bG^{\star}}^{\circ}(s)$ via $\iota^{\star}$, thus $C_{\widetilde{\bG}^{\star}}(\widetilde{s})$ is connected of type $\B_a\D_b$. With this we have the following well known result of Lusztig.
\end{pa}

\begin{lem}[Lusztig]\label{lem:cuspidal-char}
If the (geometric) series $\mathcal{E}_0(\widetilde{\bG}^F,\widetilde{s})$ contains a cuspidal character then $a=e(e+1)$ and $b=f^2$ for some non-negative integers satisfying either the condition $e\geqslant 1$ or the condition $f\geqslant 2$ and the cuspidal character is unique. If $\widetilde{\chi} \in \mathcal{E}_0(\widetilde{\bG}^F,\widetilde{s},\widetilde{\mathfrak{C}}) \subseteq \mathcal{E}_0(\widetilde{\bG}^F,\widetilde{s})$ is a cuspidal character then the corresponding family $\Irr(W_{\widetilde{\bG}^{\star}}(\widetilde{s})\mid\widetilde{\mathfrak{C}})$ contains the special character
\begin{equation*}
\begin{bmatrix}
0 & 1 & \cdots & e-1 & e\\
1 & 2 & \cdots & e
\end{bmatrix}
\boxtimes
\begin{bmatrix}
1 & 2 & \cdots & f-1 & f\\
1 & 2 & \cdots & f-1
\end{bmatrix} \in \mathbb{X}_{a,1}^{0,0} \times \mathbb{Y}_{b,1}^{0,0}.
\end{equation*}
Moreover, we have $n_{\widetilde{\chi}} = 2^{e+f-\Delta(f)}$ where $\Delta(f) = 0$ if $f=0$ and $\Delta(f) = 1$ if $f \neq 0$.
\end{lem}

\begin{proof}
Fix an element $w_1^{\star} \in T_{\widetilde{W}^{\star}}(\widetilde{s},F^{\star})$ and a Jordan decomposition $\Psi_{\widetilde{s}} : \mathcal{E}_0(\widetilde{\bG}^F,\widetilde{s}) \to \mathcal{E}_0(C_{\widetilde{\bG}^{\star}}(s)^{w_1^{\star}F^{\star}},1)$. By \cite[7.8.2]{lusztig:1977:irreducible-representations-of-finite-classical-groups} we have $\widetilde{\chi} \in \mathcal{E}_0(\widetilde{\bG}^F,\widetilde{s})$ is cuspidal if and only if $\Psi_{\widetilde{s}}(\widetilde{\chi})$ is cuspidal. Now the unipotent character $\Psi_{\widetilde{s}}(\widetilde{\chi})$ is a tensor product $\widetilde{\psi}'\boxtimes\widetilde{\psi}''$ of unipotent characters corresponding to the two factors of $C_{\widetilde{\bG}^{\star}}(s)^{w_1^{\star}F^{\star}}$ and we have $\Psi_{\widetilde{s}}(\widetilde{\chi})$ is cuspidal if and only if $\widetilde{\psi}'$ and $\widetilde{\psi}''$ are cuspidal. The statement now follows from \cite[8.2]{lusztig:1977:irreducible-representations-of-finite-classical-groups}, see also \cite[3.6.1]{lusztig:1977:irreducible-representations-of-finite-classical-groups} and \cite[8.1]{lusztig:1984:characters-of-reductive-groups}.
\end{proof}

\begin{prop}\label{prop:numerical-trick}
Recall our assumption that $\iota^{\star}(\widetilde{s}) = s$ is quasi-isolated in $\bG^{\star}$. If $\widetilde{\chi} \in \mathcal{E}_0(\widetilde{\bG}^F,\widetilde{s})$ is a cuspidal irreducible character and $u \in \mathcal{O}_{\widetilde{\chi}}^*$ is a representative of the wave front set then we have $n_{\widetilde{\chi}} = |A_{\widetilde{\bG}}(u)|$.
\end{prop}

\begin{proof}
We will denote by $j_{W_{\widetilde{\bG}^{\star}}(\widetilde{s})}^{W^{\star}}$ the $j$-induction of characters from $W_{\widetilde{\bG}^{\star}}(\widetilde{s})$ to $W^{\star}$ with respect to the natural module of $W^{\star}$. Let $E = \ssymb{A}{B} \boxtimes \ssymb{C}{D} \in \mathbb{X}_{a,1}^{0,0} \times \mathbb{Y}_{b,1}^{0,0}$ be the special character of $W_{\widetilde{\bG}^{\star}}(\widetilde{s})$ described in \cref{lem:cuspidal-char}. Note that $E$ is invariant under tensoring with the sign character because it is contained in a cuspidal family, see \cite[8.1]{lusztig:1984:characters-of-reductive-groups}. In particular, $j_{W_{\widetilde{\bG}^{\star}}(\widetilde{s})}^{W^{\star}}(E\otimes\sgn) = j_{W_{\widetilde{\bG}^{\star}}(\widetilde{s})}^{W^{\star}}(E)$ and according to \cite[5.3(b)]{lusztig:2009:unipotent-classes-and-special-Weyl} we have
\begin{equation}\label{eq:springer-map-C}
j_{W_{\widetilde{\bG}^{\star}}(\widetilde{s})}^{W^{\star}}(E) = \symb{A}{B} \oplus \symb{C}{D}.
\end{equation}
Combining this with the description of the map in \cref{eq:spr-map} allows us to compute $\Spr(j_{W_{\widetilde{\bG}^{\star}}(\widetilde{s})}^{W^{\star}}(E)) \in \mathbb{X}_{n,1}^{1,1}$.

The symbol $\Spr(j_{W_{\widetilde{\bG}^{\star}}(\widetilde{s})}^{W^{\star}}(E))$ corresponds to a pair $(\mathcal{O},\mathscr{E}) \in \mathcal{N}_{\bG}$ under the map mentioned in \cref{eq:unipotent-pairs} and $\mathcal{O}$ is precisely the wave front set $\mathcal{O}_{\widetilde{\chi}}^*$, see \cite[\S10]{lusztig:1992:a-unipotent-support} and \cite[\S12]{taylor:2016:GGGRs-small-characteristics}. Assume $c_0 \leqslant c_1 \leqslant \cdots \leqslant c_s$ are the entries of the symbol $\Spr(j_{W_{\widetilde{\bG}^{\star}}(\widetilde{s})}^{W^{\star}}(E))$. If all these entries are distinct then the partition parameterising $\mathcal{O}_{\widetilde{\chi}}^*$, c.f., \cref{pa:uni-classes}, is obtained from the sequence $(2c_0,2c_1-2,\dots,2c_s-2s)$ by removing any zero entries, see \cite[\S2.B]{geck-malle:2000:existence-of-a-unipotent-support}. The general case is more complicated but we will not need it here, see \cite[\S2.B]{geck-malle:2000:existence-of-a-unipotent-support} for details.

Using the above we may now prove the statement. There are two cases to consider. Firstly, assume $0 \leqslant e < f$ and let $k = f - e > 0$ then by \cref{eq:spr-map,eq:springer-map-C} we have $\Spr(j_{W_{\widetilde{\bG}^{\star}}(\widetilde{s})}^{W^{\star}}(E))$ is the symbol
\begin{equation*}
\kbordermatrix{
& 0 & 1 & \cdots & k-1  & k    & k+1  & \cdots & k+e-1   & k+e\\
& 0 & 3 & \cdots & 3k-3 & 3k   & 3k+4 & \cdots & 3k+4e-4 & 3k+4e\\
& 1 & 4 & \cdots & 3k-2 & 3k+2 & 3k+6 & \cdots & 3k+4e-2
}
\end{equation*}
The unipotent class corresponding to this symbol is thus parameterised by the partition $\lambda = 2\mu$ where
\begin{equation*}
\mu = (1,1,2,2,\dots,k-1,k-1,k,k+1,\dots,k+2e)
\end{equation*}
As $\lambda$ contains $(k-1) + (2e+1) = e+f$ distinct even numbers and one such number occurs an odd number of times we have $n_{\widetilde{\chi}} = 2^{e+f-1} = |A_{\widetilde{\bG}}(u)|$, c.f., \cref{lem:comp-grp,lem:cuspidal-char}.

Assume $0 \leqslant f \leqslant e$ and let $k = e - f \geqslant 0$ then by \cref{eq:spr-map,eq:springer-map-C} we have $\Spr(j_{W^{\star}(s)}^{W^{\star}}(E))$ is the symbol
\begin{equation*}
\kbordermatrix{
& 0 & 1 & \cdots & k-1  & k    & k+1  & \cdots & k+f-1   & k+f\\
& 0 & 3 & \cdots & 3k-3 & 3k   & 3k+4 & \cdots & 3k+4f-4 & 3k+4f\\
& 2 & 5 & \cdots & 3k-2 & 3k+2 & 3k+6 & \cdots & 3k+4f-2
}
\end{equation*}
The unipotent class corresponding to this symbol is thus parameterised by the partition $\lambda = 2\mu$ where
\begin{equation*}
\mu = (1,1,2,2,...,k,k,k+1,\dots,k+2f)
\end{equation*}
As $\lambda$ contains $k + 2f = e+f$ distinct even numbers and one such number occurs an odd number of times if and only if $f \neq 0$ we have $n_{\widetilde{\chi}} = 2^{e+f-\Delta(f)} = |A_{\widetilde{\bG}}(u)|$, c.f., \cref{lem:comp-grp,lem:cuspidal-char}.
\end{proof}

\begin{thm}\label{thm:cuspidal-fixed}
Assume $(s,a^{\star}) \in \mathcal{T}_{\bG^{\star}}(\bT_0^{\star},F^{\star})$ is such that $s$ is quasi-isolated in $\bG^{\star}$ then ${}^{\sigma}\mathcal{E}_0(\bG^F,s,a^{\star}) = \mathcal{E}_0(\bG^F,s,a^{\star})$. Moreover, if $\chi \in \mathcal{E}_0(\bG^F,s,a^{\star})$ is a cuspidal character then ${}^{\sigma}\chi = \chi$.
\end{thm}

\begin{proof}
Firstly, as $s^2 = 1$ we clearly have $\sigma^{\star}(s) = s$ and as $\sigma^{\star}$ acts trivially on $W^{\star}$ we see that $\sigma^{\star}$ fixes $(s,a^{\star}) \in \mathcal{T}_{\bG^{\star}}(\bT_0^{\star},F^{\star})$; this implies ${}^{\sigma}\mathcal{E}_0(\bG^F,s,a^{\star}) = \mathcal{E}_0(\bG^F,s,a^{\star})$ by \cref{pa:stab-Lusztig-series}. Now let $(\widetilde{s},a^{\star}) \in \mathcal{T}_{\widetilde{\bG}^{\star}}(\widetilde{\bT}_0^{\star},F^{\star})$ be such that $\iota^{\star}(\widetilde{s},a^{\star}) = (s,a^{\star})$ and let $\widetilde{\chi} \in \mathcal{E}_0(\widetilde{\bG}^F,\widetilde{s}) = \mathcal{E}_0(\widetilde{\bG}^F,\widetilde{s},a^{\star})$ be a character such that $\chi$ occurs in the restriction of $\widetilde{\chi}$, c.f., \cref{sec:regular-embeddings}. By \cite[12.1]{bonnafe:2006:sln} we have $\chi$ is cuspidal if and only if $\widetilde{\chi}$ is cuspidal. As $\widetilde{\chi}$ is the unique cuspidal character contained in $\mathcal{E}_0(\widetilde{\bG}^F,\widetilde{s})$, c.f., \cref{lem:cuspidal-char}, we see that the cuspidal characters contained in $\mathcal{E}_0(\bG^F,s,a^{\star})$ are precisely the irreducible constituents of $\Res_{\bG^F}^{\widetilde{\bG}^F}(\widetilde{\chi})$. In particular, they are all conjugate under $\widetilde{\bG}^F$.

Let $\mathcal{O} = \mathcal{O}_{\widetilde{\chi}}^*$ denote the wave front set of $\widetilde{\chi}$ and let $u_1,\dots,u_r$ be representatives for the $\widetilde{\bG}^F$-conjugacy classes contained in $\mathcal{O}^F$. We have a character of $\widetilde{\bG}^F$
\begin{equation*}
\widetilde{\Gamma}_{\mathcal{O}} = \widetilde{\Gamma}_{u_1} + \cdots + \widetilde{\Gamma}_{u_r}
\end{equation*}
obtained by summing the corresponding characters of the GGGRs $\widetilde{\Gamma}_{u_i} := \Gamma_{u_i}^{\widetilde{\bG}^F}$. By \cite[2.2, 2.4]{taylor:2013:on-unipotent-supports} we have $F$ acts trivially on the component group $A_{\widetilde{\bG}}(u_i)$ for all $1 \leqslant i \leqslant r$. In particular, by \cite[15.4]{taylor:2016:GGGRs-small-characteristics} we have
\begin{equation*}
\langle \widetilde{\Gamma}_{\mathcal{O}}, \widetilde{\chi}\rangle_{\widetilde{\bG}^F} = \frac{|A_{\widetilde{\bG}}(u_i)|}{n_{\widetilde{\chi}}} = 1
\end{equation*}
where the second equality follows from \cref{prop:numerical-trick}. As the $\widetilde{\Gamma}_{u_i}$ are characters there must exist a unique GGGR $\widetilde{\Gamma}_u = \widetilde{\Gamma}_{u_i}$, with $1 \leqslant i \leqslant r$, such that $\langle \widetilde{\Gamma}_u, \widetilde{\chi}\rangle_{\widetilde{\bG}^F} = 1$.

The unipotent element $u$ is contained in $\bG^F$ and we have a corresponding GGGR $\Gamma_u := \Gamma_u^{\bG^F}$ of $\bG^F$. This GGGR has the property that $\widetilde{\Gamma}_u = \Ind_{\bG^F}^{\widetilde{\bG}^F}(\Gamma_u)$ so consequently
\begin{equation*}
1 = \langle \widetilde{\Gamma}_u, \widetilde{\chi}\rangle_{\widetilde{\bG}^F} = \langle \Ind_{\bG^F}^{\widetilde{\bG}^F}(\Gamma_u), \widetilde{\chi}\rangle_{\widetilde{\bG}^F} = \langle \Gamma_u, \Res_{\bG^F}^{\widetilde{\bG}^F}(\widetilde{\chi})\rangle_{\bG^F}.
\end{equation*}
In particular, the restriction $\Res_{\bG^F}^{\widetilde{\bG}^F}(\widetilde{\chi})$ has a unique irreducible constituent $\chi_0$ such that $\langle \Gamma_u,\chi_0\rangle_{\bG^F} = 1$. It's easily seen that ${}^{\sigma}\chi_0 \in {}^{\sigma}\mathcal{E}_0(\bG^F,s,a^{\star}) = \mathcal{E}_0(\bG^F,s,a^{\star})$ is also cuspidal so is a constituent of $\Res_{\bG^F}^{\widetilde{\bG}^F}(\widetilde{\chi})$. Moreover, we have
\begin{equation*}
1 = \langle \Gamma_u, \chi_0\rangle_{\bG^F} = \langle {}^{\sigma}\Gamma_u, {}^{\sigma}\chi_0\rangle_{\bG^F} = \langle \Gamma_u, {}^{\sigma}\chi_0\rangle_{\bG^F}
\end{equation*}
by \cref{prop:invariant-GGGRs,lem:conj-unipotent-elms}. As $\chi_0$ is uniquely determined by this property we must have ${}^{\sigma}\chi_0 = \chi_0$.

If the restriction $\Res_{\bG^F}^{\widetilde{\bG}^F}(\widetilde{\chi})$ is irreducible then we have $\chi_0 = \chi$ and we're done, so we may assume this is not the case. As $\bG^F$ is a symplectic group we have $\Res_{\bG^F}^{\widetilde{\bG}^F}(\widetilde{\chi}) = \chi_0 + \chi_1$ with $\chi_1 \in \Irr(\bG^F)$, so either $\chi = \chi_0$ or $\chi = \chi_1$. If $\chi = \chi_0$ we're done. Assume $\chi = \chi_1$ and ${}^{\sigma}\chi = \chi_0$ then ${}^{\sigma^2}\chi = {}^{\sigma}\chi_0 = \chi_0 = {}^{\sigma}\chi$ which implies ${}^{\sigma}\chi = \chi$, a contradiction. Therefore we have ${}^{\sigma}\chi = \chi$ as desired.
\end{proof}

\section{Quasi-Isolated Series}\label{sec:quasi-isolated-series}
\begin{pa}\label{pa:std-levi-para}
We now wish to understand the effect of the field automorphism $\sigma$ on a geometric Lusztig series $\mathcal{E}_0(\bG^F,s)$ when $s \in\bT_0^{\star}$ is quasi-isolated in $\bG^{\star}$. We already know, by \cref{thm:cuspidal-fixed}, that ${}^{\sigma}\mathcal{E}_0(\bG^F,s) = \mathcal{E}_0(\bG^F,s)$ and that any cuspidal character contained in $\mathcal{E}_0(\bG^F,s)$ is $\sigma$-fixed. Using this we will now show that all the characters in $\mathcal{E}_0(\bG^F,s)$ are $\sigma$-fixed. To do this we will use Harish-Chandra theory. With this in mind, let $I \subseteq \Delta$ be an $F$-stable subset of simple roots then the corresponding parabolic and Levi subgroups $\bL_I \leqslant \bP_I$ are $F$-stable. We have a corresponding Harish-Chandra induction map $R_I^{\bG^F} = R_{\bL_I^F \subseteq \bP_I^F}^{\bG^F}$ which is defined to be the composition of inflation from $\bL_I^F$ to $\bP_I^F$ followed by induction from $\bP_I^F$ to $\bG^F$. Note, this coincides with the map $R_{I,1}^{\bG}$ defined in \cref{pa:twisted-induction}.
\end{pa}

\begin{pa}
If $\lambda \in \Irr(\bL_I^F)$ is a cuspidal irreducible character then we denote by $\mathcal{E}(\bG^F,I,\lambda) \subseteq \Irr(\bG^F)$ the irreducible constituents of $R_I^{\bG^F}(\lambda)$. Moreover, we denote by $\Cusp(\bG^F)$ the set of pairs $(I,\lambda)$ consisting of an $F$-stable subset $I \subseteq \Delta$ of simple roots and a cuspidal character $\lambda \in \Irr(\bL_I^F)$. We define an equivalence relation $\sim$ on $\Cusp(\bG^F)$ by setting $(I,\lambda)\sim(J,\mu)$ if there exists an element $w \in W^F$ such that $w(I) = J$ and ${}^g\lambda = \mu$; where $g \in N_{\bG}(\bT_0)^F$ represents $w$. Note, such an element exists by the Lang--Steinberg Theorem and satisfies ${}^g\bL_I = \bL_{w(I)} = \bL_J$. We then have a disjoint union
\begin{equation*}
\Irr(\bG^F) = \bigsqcup_{[I,\lambda] \in \Cusp(\bG^F)/\sim} \mathcal{E}(\bG^F,I,\lambda),
\end{equation*}
where the union runs over the corresponding equivalence classes. The sets $\mathcal{E}(\bG^F,I,\lambda)$ are called Harish-Chandra series.
\end{pa}

\begin{pa}
By \cite[11.10]{bonnafe:2006:sln} it is known that any Lusztig series $\mathcal{E}_0(\bG^F,s,a^{\star})$ is a disjoint union of Harish-Chandra series. Thus, it suffices to show that for each Harish-Chandra series $\mathcal{E}(\bG^F,I,\lambda) \subseteq \mathcal{E}_0(\bG^F,s)$ we have ${}^{\sigma}\mathcal{E}(\bG^F,I,\lambda) = \mathcal{E}(\bG^F,I,\lambda)$ and, moreover, each character in $\mathcal{E}(\bG^F,I,\lambda)$ is $\sigma$-fixed. As $\sigma(\bL_I) = \bL_{\sigma(I)}$ and $\sigma(\bP_I) = \bP_{\sigma(I)}$ we clearly have ${}^{\sigma}R_I^{\bG^F}(\lambda) = R_{\sigma(I)}^{\bG^F}({}^{\sigma}\lambda)$. As $\sigma$ is a split Frobenius endomorphism we have $\sigma(I) = I$ so it suffices to show ${}^{\sigma}\lambda = \lambda$. To do this, we need to describe the possible Harish-Chandra series contained in $\mathcal{E}_0(\bG^F,s)$. This amounts to the following well known result.
\end{pa}

\begin{lem}\label{lem:cusp-Levi}
Assume $s^2 = 1$ and $(I,\lambda) \in \Cusp(\bG^F)$ satisfies $\mathcal{E}(\bG^F,I,\lambda) \subseteq \mathcal{E}_0(\bG^F,s)$ then the derived subgroup of $\bL_I$ is $\bL_I' = \Sp_{2k}(\mathbb{K})$, where $k = e(e+1)$ for some $e \geqslant 0$. Moreover, we have $\bL_I = \bL_I' \times Z^{\circ}(\bL_I)$ and $\lambda = \theta\boxtimes\psi$ with $\psi \in \mathcal{E}_0(\bL_I'^F,s')$ a cuspidal character, $s'^2 = 1$, and $\theta \in \Irr(Z^{\circ}(\bL_I)^F)$ satisfies $\theta^2 = 1$.
\end{lem}

\begin{proof}
The Levi subgroup $\bL_I^F$ is a direct product $\prod_{i=1}^t \GL_{n_i}(q) \times \Sp_{2k}(q)$, embedded diagonally in $\bG^F$, with $n = k + n_1 + \cdots + n_t$; see \cite[1.7.3]{geck:2003:intro-to-algebraic-geometry}. Moreover, the dual group $\bL_I^{\star F^{\star}}$ is also a direct product $\prod_{i=1}^t \GL_{n_i}(q) \times \SO_{2k+1}(q)$ embedded diagonally in $\bG^{\star F^{\star}}$. The fact that $n_1 = \cdots = n_t = 1$ follows from \cite[7.12]{lusztig:1977:irreducible-representations-of-finite-classical-groups} together with our assumption that $s^2 = 1$; the rest of the statement is easy.
\end{proof}

\begin{pa}\label{pa:HC-setup}
Let us continue with the notation and assumptions of \cref{lem:cusp-Levi}. It follows from \cref{thm:cuspidal-fixed} that ${}^{\sigma}\psi = \psi$. If $r$ is such that $\sigma = F_r$ then we have ${}^{\sigma}\theta = \theta^r = \theta$ because $\theta^2 = 1$ and $p$ is odd. In particular, we have ${}^{\sigma}\lambda = \lambda$ so ${}^{\sigma}R_I^{\bG^F}(\lambda) = R_I^{\bG^F}(\lambda)$. By \cite{geck:1993:a-note-on-harish-chandra} and \cite[8.6]{lusztig:1984:characters-of-reductive-groups}, there exists an extension $\widetilde{\lambda} \in \Irr(N_{\bG^F}(\bL_I,\lambda))$ of $\lambda$ to its stabiliser in $N_{\bG^F}(\bL_I)$. Clearly $\sigma(N_{\bG^F}(\bL_I,\lambda)) = N_{\bG^F}(\bL_I,{}^{\sigma}\lambda) = N_{\bG^F}(\bL_I,\lambda)$ and so ${}^{\sigma}\widetilde{\lambda} \in \Irr(N_{\bG^F}(\bL_I,\lambda))$ is also an extension of $\lambda$. To study the action of $\sigma$ on $\mathcal{E}(\bG^F,I,\lambda)$ we need to understand how $\sigma$ acts on $\widetilde{\lambda}$ and for this we need to understand $N_{\bG^F}(\bL_I)$. We decompose $\bL_I$ as $Z^{\circ}(\bL_I) \times \bL_I'$, as in \cref{lem:cusp-Levi}. Recall that the group $\bL_I$ embeds naturally in a subsystem subgroup $\bM \times \bL_I' = \Sp_{2(n-k)}(\mathbb{K}) \times \Sp_{2k}(\mathbb{K}) \leqslant \bG$ such that $Z^{\circ}(\bL_I)$ is a maximal torus of $\bM$. This subgroup is $F$-stable and the corresponding subgroup of $F$-fixed points is $\Sp_{2(n-k)}(q) \times \Sp_{2k}(q)$.
\end{pa}

\begin{pa}\label{pa:sigma-stab}
We claim that $N_{\bG^F}(\bL) = \bL_I'^F \times N_{\bM^F}(Z^{\circ}(\bL_I))$. If $I \subseteq \Delta$ is as in \cref{pa:HC-setup} then we have $W_{\bG}(\bL_I)$ is isomorphic to $N_W(W_I)/W_I$, where $W_I \leqslant W$ is the corresponding parabolic subgroup. Furthermore, this isomorphism is compatible with the induced actions of $F$. As $F$ acts trivially on $W$ this implies that we have
\begin{equation*}
N_W(W_I)/W_I = (N_W(W_I)/W_I)^F \cong (N_{\bG}(\bL_I)/\bL_I)^F \cong N_{\bG}(\bL_I)^F/\bL_I^F,
\end{equation*}
where the last isomorphism holds because $F$ is a Frobenius endomorphism and $\bL_I$ is connected, see \cite[3.13]{digne-michel:1991:representations-of-finite-groups-of-lie-type}. Now the group $N_W(W_I)/W_I$ is well known to be a Coxeter group of type $\B_{n-k}$, see \cite{howlett:1980:normalizers-of-parabolic-subgroups} for instance. From this the claim follows. We may now prove the following.
\end{pa}

\begin{lem}\label{lem:sigma-fixed-extension}
Any extension $\widetilde{\lambda} \in \Irr(N_{\bG^F}(\bL_I,\lambda))$ of $\lambda$ satisfies ${}^{\sigma}\widetilde{\lambda} = \widetilde{\lambda}$.
\end{lem}

\begin{proof}
From \cref{pa:sigma-stab} we see that the stabiliser $N_{\bG^F}(\bL_I,\lambda)$ is simply $\bL_I'^F \times N_{\bM^F}(Z^{\circ}(\bL_I),\theta)$ so the extension $\widetilde{\lambda}$ is of the form $\psi \boxtimes \widetilde{\theta}$ where $\widetilde{\theta} \in \Irr(N_{\bM^F}(Z^{\circ}(\bL_I),\theta))$ is an extension of $\theta$. As ${}^{\sigma}\psi = \psi$ we need only show that ${}^{\sigma}\widetilde{\theta} = \widetilde{\theta}$. Now, as $\theta$ is a linear character so is $\widetilde{\theta}$ which implies that $\widetilde{\theta}$ is a homomorphism. Each coset $nZ^{\circ}(\bL_I) \in W_{\bM}(Z^{\circ}(\bL_I))$ is $F_p$-stable and thus contains an $F_p$-fixed point. These coset representatives are then also $\sigma$-fixed and $F$-fixed. This implies that $N_{\bM^F}(Z^{\circ}(\bL_I))$ is generated by $Z^{\circ}(\bL_I)^F$ and a set of $\sigma$-fixed coset representatives. As ${}^{\sigma}\widetilde{\theta}(t) = {}^{\sigma}\theta(t) = \theta(t) = \widetilde{\theta}(t)$ for any $t \in Z^{\circ}(\bL_I)^F$ and clearly ${}^{\sigma}\widetilde{\theta}(n) = \widetilde{\theta}(n)$ for any $n \in N_{\bM^{F_p}}(\bT_0) \leqslant N_{\bM^F}(\bT_0)$ we must have ${}^{\sigma}\widetilde{\theta} = \widetilde{\theta}$ because it's a homomorphism.
\end{proof}

\begin{pa}
We are now ready to prove the main result of this article. We note that by \cref{lem:DL-geo-series-compare,pa:geometric-series} the statement given below is equivalent to the statement given in \cref{thm:main-theorem}.
\end{pa}

\begin{thm}\label{prop:quasi-iso-fixed}
Recall our assumption that $\bG$ is the symplectic group $\Sp_{2n}(\mathbb{K})$ and $p$ is odd. Assume $s \in \bT_0^{\star}$ is quasi-isolated in $\bG^{\star}$ then ${}^{\sigma}\chi = \chi$ for all $\chi \in \mathcal{E}_0(\bG^F,s)$.
\end{thm}

\begin{proof}
It suffices to show that each element of a Harish-Chandra series $\mathcal{E}(\bG^F,I,\lambda) \subseteq \mathcal{E}_0(\bG^F,s)$ is fixed by $\sigma$. As discussed in \cref{pa:HC-setup} we have ${}^{\sigma}\lambda = \lambda$ and by \cref{lem:sigma-fixed-extension} any extension $\widetilde{\lambda} \in \Irr(N_{\bG^F}(\bL_I,\lambda))$ of $\lambda$ satisfies ${}^{\sigma}\widetilde{\lambda} = \widetilde{\lambda}$. If $W_{\bG^F}(\bL_I,\lambda) = N_{\bG^F}(\bL_I,\lambda)/\bL_I^F$ then we have a bijection $\Irr(W_{\bG^F}(\bL_I,\lambda)) \to \mathcal{E}(\bG^F,I,\lambda)$, denoted $\eta \mapsto R_I^{\bG^F}(\lambda)_{\eta}$, and by \cite[5.6]{malle-spaeth:2016:characters-of-odd-degree} this bijection can be chosen such that
\begin{equation*}
{}^{\sigma}R_I^{\bG^F}(\lambda)_{\eta} = R_I^{\bG^F}({}^{\sigma}\lambda)_{{}^{\sigma}\eta} = R_I^{\bG^F}(\lambda)_{{}^{\sigma}\eta}.
\end{equation*}
Note that the linear character denoted $\delta_{\lambda,\sigma}$ in \cite[5.6]{malle-spaeth:2016:characters-of-odd-degree} is trivial in our case because ${}^{\sigma}\widetilde{\lambda} = \widetilde{\lambda}$ is both an extension of $\lambda$ and ${}^{\sigma}\lambda$. The same argument as that used in \cref{pa:sigma-stab} shows that $\sigma$ is the identity on $W_{\bG^F}(\bL_I,\lambda)$ so ${}^{\sigma}\eta = \eta$ and the statement follows.
\end{proof}

\section{On the Inductive McKay Condition}
\begin{pa}
We now wish to show that the global portion of Sp\"ath's criterion for the inductive McKay condition holds for the irreducible characters of $\bG^F$, see \cite[Theorem 2.12]{spaeth:2012:inductive-mckay-defining}. Specifically let us recall from \cref{pa:autos-symp} that the semidirect product $\widetilde{\bG}^F \rtimes D$ acts on the set of irreducible characters $\Irr(\bG^F)$ and all automorphisms of $\bG^F$ appear in this action. With this we have the following; recall the notation of \cref{pa:stab-notation}.
\end{pa}

\begin{thm}\label{prop:mckay}
For any $\chi \in \Irr(\bG^F)$ we have $(\widetilde{\bG}^F \rtimes D)_{\chi} = \widetilde{\bG}^F_{\chi} \rtimes D_{\chi}$.
\end{thm}

\begin{proof}
Let us assume that $\chi \in \mathcal{E}_0(\bG^F,s,a^{\star})$. We need to show that for any $g \in \widetilde{\bG}^F$ and $\sigma \in D$ we have ${}^{\sigma g}\chi = \chi$ if and only if ${}^g\chi = \chi$ and ${}^{\sigma}\chi = \chi$. Assume first that $\sigma$ is such that ${}^{\sigma}\mathcal{E}_0(\bG^F,s,a^{\star}) \neq \mathcal{E}_0(\bG^F,s,a^{\star})$. For any $g \in \widetilde{\bG}^F$ we have ${}^g\chi \in \mathcal{E}_0(\bG^F,s,a^{\star})$ which means ${}^{\sigma g}\chi \in {}^{\sigma}\mathcal{E}_0(\bG^F,s,a^{\star})$. As ${}^{\sigma g}\chi$ and $\chi$ are contained in different series we must have ${}^{\sigma g}\chi \neq \chi$ so there is nothing to check in this case.

Now assume ${}^{\sigma}\mathcal{E}_0(\bG^F,s,a^{\star}) = \mathcal{E}_0(\bG^F,s,a^{\star})$ then we are in the setup of \cref{sec:equivariance-twisted-induction}. After possibly replacing the pair $(s,a^{\star}) \in \mathcal{T}_{\bG^{\star}}(\bT_0^{\star},F^{\star})$ by another pair in the same $W^{\star}$-orbit we may assume that the Levi cover of $C_{\bG^{\star}}(s)$ is a standard Levi subgroup $\bL_I^{\star}$. We can then decompose $\bL_I^{\star}$ as a product $\bM_{\C}^{\star} \times \bM_{\A}^{\star}$ such that $\bM_{\C}^{\star}$ is a special orthogonal group and $\bM_{\A}^{\star}$ is a product of general linear groups. Note that these general linear groups may be tori, i.e., isomorphic to $\GL_1(\mathbb{K})$. Correspondingly, in $\bG$, we have $\bL_I = \bM_{\C} \times \bM_{\A}$ where $\bM_{\C}$ is a symplectic group and $\bM_{\A}$ is a direct product of general linear groups.

As $F^{\star}$ acts trivially on $\Delta^{\star}$ we have $T_{W^{\star}}(I^{\star},F^{\star}) = N_{W^{\star}}(I^{\star})$. Let us denote by $J \subseteq \Delta$, resp., $K \subseteq \Delta$, the set of simple roots such that the root subgroups $\{\bX_{\pm\alpha} \mid \alpha \in J\}$, resp., $\{\bX_{\pm\alpha} \mid \alpha \in K\}$, generate the derived subgroup of $\bM_{\C}$, resp., $\bM_{\A}$. An easy argument shows that for any element $x^{\star} \in N_{W^{\star}}(I^{\star})$ we have $x^{\star}(J^{\star}) = J^{\star}$ and moreover $x^{\star}$ must fix pointwise any root in $J^{\star}$; simply because $J^{\star}$ admits no graph automorphisms. In fact, for each element $x^{\star} \in N_{W^{\star}}(I^{\star})$ we can find a representative $n_{x^{\star}} \in N_{\bG^{\star}}(\bT_0^{\star}) \cap C_{\bG^{\star}}(\bM_{\C}^{\star})$. This follows from the fact that $\bL_I^{\star}$ is contained in a subgroup $\widehat{\bL}_I^{\star} = \bM_{\C}^{\star} \times \widehat{\bM}_{\A}^{\star}$, where $\widehat{\bM}_{\A}^{\star}$ is a full orthogonal group, together with \cite{howlett:1980:normalizers-of-parabolic-subgroups}.

Now assume $w_1^{\star} \in T_{W^{\star}}(I^{\star},F^{\star})$ and $z^{\star} \in T_{W^{\star}}(I^{\star},F^{\star})$ are as in \cref{pa:stab-Lusztig-series-1}. Choosing the representatives as above means that we have $\bL_I^{\star w_1^{\star}F^{\star}} = \bM_{\C}^{\star F^{\star}} \times \bM_{\A}^{\star w_1^{\star}F^{\star}}$. In $\bG$ we have $\bL_I$ is contained in a subgroup $\widehat{\bL}_I = \bM_{\C} \times \widehat{\bM}_{\A}$, where $\widehat{\bM}_{\A}$ is a symplectic group. This is a subsystem subgroup as in \cref{pa:HC-setup}. From our choice of representatives $n_w \in N_{\bG}(\bT_0)$, c.f., \cref{pa:weyl-group-reps}, we see that for any element $w \in Z_W(I,F) = N_W(I)$ we have $n_w \in N_{\bG}(\bT_0) \cap \widehat{\bM}_{\A} \leqslant C_{\bG}(\bM_{\C})$. Therefore, we similarly have $\bL_I^{Fw_1} = \bM_{\C}^F \times \bM_{\A}^{Fw_1}$.

If $\chi \in \mathcal{E}_0(\bG^F,s,a^{\star})$ then by \cref{thm:Lusztig-bijection} we have $\chi = (-1)^{\ell(w_1)}R_{I,w_1}^{\bG}(\psi)$ for a unique character $\psi \in \mathcal{E}_0(\bL_I^{Fw_1},s,b^{\star})$ where $b^{\star} = a^{\star}w_1^{\star-1}$. Let us write the element $s$ as a product $s_{\C}s_{\A}$ with $s_{\C} \in \bM_{\C}^{\star} \cap \bT_0^{\star}$ and $s_{\A} \in \bM_{\A}^{\star} \cap \bT_0^{\star}$. We then have $W^{\star}(s) = W_J^{\star}(s_{\C}) \times W_K^{\star}(s_{\A})$ and $W^{\star\circ}(s) = W_J^{\star\circ}(s_{\C}) \times W_K^{\star}(s_{\A})$ because $\bL_K$ has a connected centre. This means we can identify $\mathcal{A}_{W^{\star}}(s)$ and $\mathcal{A}_{W_J^{\star}}(s_{\C})$ and so also $\mathcal{A}_{W^{\star}}(s,F^{\star})$ and $\mathcal{A}_{W_J^{\star}}(s_{\C},F^{\star})$. With this in mind we have
\begin{equation*}
\mathcal{E}_0(\bL_I^{Fw_1},s,b^{\star}) = \mathcal{E}_0(\bM_{\C}^F,s_{\C},b^{\star})\otimes \mathcal{E}_0(\bM_{\A}^{Fw_1},s_{\A}).
\end{equation*}
Here $\otimes$ denotes tensor product of characters. Moreover, we have used that the rational and geometric Lusztig series of $\bM_{\A}^{Fw_1}$ coincide, see \cref{pa:geometric-series}.

With this we can decompose $\psi$ uniquely as a tensor product $\psi_{\C} \otimes \psi_{\A}$ with $\psi_{\C} \in \mathcal{E}_0(\bM_{\C}^F,s_{\C},b^{\star})$ and $\psi_{\A} \in \mathcal{E}_0(\bM_{\A}^{Fw_1},s_{\A})$. Recall the element $z \in W$ is such that $z(I) = z\sigma(I) = I$. To apply \cref{thm:equivariance-DL-induction} we must find an element $n \in N_{\bG}(\bT_0)$ such that $nF(\sigma(n_{w_1})n^{-1}) = nn_{w_1}F(n^{-1}) = n_{w_1}$. Inspecting the proof of \cref{prop:equivariance} we easily see that $n$ may be chosen to lie in $C_{\bG}(\bM_{\C})$ because $n_zF(\sigma(n_{w_1})n_z^{-1}n_{w_1}^{-1}) \in \bT_0 \cap \widehat{\bM}_{\A} \leqslant C_{\bG}(\bM_{\C})$ and this intersection is connected, which means the Lang--Steinberg theorem can be applied to it. Assuming $n$ is chosen in this way we have
\begin{equation*}
{}^{\sigma}\chi = \chi \Leftrightarrow {}^{\sigma}R_{I,w_1}^{\bG}(\psi) = R_{I,w_1}^{\bG}(\psi) \Leftrightarrow {}^{n\sigma}\psi = \psi \Leftrightarrow {}^{\sigma}\psi_{\C} \otimes {}^{n\sigma}\psi_{\A} = \psi_{\C} \otimes \psi_{\A}.
\end{equation*}
By \cref{prop:quasi-iso-fixed} we have ${}^{\sigma}\psi_{\C} = \psi_{\C}$ so ${}^{\sigma}\chi = \chi$ if and only if ${}^{n\sigma}\psi_{\A} = \psi_{\A}$.

Recall that any element of the quotient $\widetilde{\bG}^F/\bG^F$ can be represented by an element $t \in \widetilde{\bT}_0^F$. Assume $t$ is such an element and let $m \in \bT_0$ be such that $mtn_{w_1}t^{-1}F(m^{-1}) = n_{w_1}$. Applying \cref{thm:equivariance-DL-induction} we have
\begin{equation*}
{}^t\chi = \chi \Leftrightarrow {}^tR_{I,w_1}^{\bG}(\psi) = R_{I,w_1}^{\bG}(\psi) \Leftrightarrow {}^{mt}\psi = \psi \Leftrightarrow {}^{mt}\psi_{\C} \otimes {}^{mt}\psi_{\A} = \psi_{\C} \otimes \psi_{\A}.
\end{equation*}
It's clear that $mt$ induces a diagonal automorphism on $\bM_{\A}^{Fw_1}$ but as this group is a product of finite general linear and finite general unitary groups such an automorphism must be inner. This means ${}^{mt}\psi_{\A} = \psi_{\A}$ so we have ${}^t\chi = \chi$ if and only if ${}^{mt}\psi_{\C} = \psi_{\C}$. Finally we have
\begin{equation*}
{}^{t\sigma}\chi = \chi \Leftrightarrow {}^{t\sigma}R_{I,w_1}^{\bG}(\psi) = R_{I,w_1}^{\bG}(\psi) \Leftrightarrow {}^{mtn\sigma}\psi = \psi \Leftrightarrow {}^{mt}\psi_{\C} \otimes {}^{n\sigma}\psi_{\A} = \psi_{\C} \otimes \psi_{\A},
\end{equation*}
which holds if and only if ${}^{mt}\psi_{\C} = \psi_{\C}$ and ${}^{n\sigma}\psi_{\A} = \psi_{\A}$. This latter condition is equivalent to ${}^t\chi = \chi$ and ${}^{\sigma}\chi = \chi$ which completes the proof.
\end{proof}

\begingroup
\setstretch{0.96}
\renewcommand*{\bibfont}{\small}
\printbibliography
\endgroup
\end{document}